\documentclass[12pt]{amsart}
\usepackage{latexsym,amscd, amsmath, amssymb,epsfig, tikz, array, multirow
} 
\usetikzlibrary{patterns,arrows,decorations.pathreplacing,fit}
\pgfdeclarelayer{bg}
\pgfsetlayers{bg,main}

\theoremstyle{plain}
  \newtheorem{theorem}{Theorem}[section]
  \newtheorem{proposition}[theorem]{Proposition}
  
  \newtheorem{corollary}[theorem]{Corollary}
  
\theoremstyle{definition}

\theoremstyle{remark}

\numberwithin{equation}{section}

\def\tempbaselines
{\baselineskip14pt\lineskip2pt
   \lineskiplimit2pt}
\def\diagram#1{\null\,\vcenter{\tempbaselines
\mathsurround=0pt
    \ialign{\hfil$##$\hfil&&\quad\hfil$##$\hfil\crcr
      \mathstrut\crcr\noalign{\kern-\baselineskip}
  #1\crcr\mathstrut\crcr\noalign{\kern-\baselineskip}}}\,}
\def\clap#1{\hbox to 0pt{\hss$#1$\hss}}

\def \End{\mathop{\rm End}\nolimits} 
\def \Ext{\mathop{\rm Ext}\nolimits} 

\def \Hom{\mathop{\rm Hom}\nolimits}

\def \Rad{\mathop{\rm Rad}\nolimits}

\def \Soc{\mathop{\rm Soc}\nolimits}

\def\FF{{\mathbb F}}

\begin{document}

\title[Young Modules and Schur algebras]
{The Structure of Young Modules and Schur algebras in Characteristic 2}

\author{Moriah Elkin}\email{elkin048@umn.edu}\author{Peter Webb}
\email{webb@math.umn.edu}
\address{School of Mathematics\\
University of Minnesota\\
Minneapolis, MN 55455, USA}


\thanks{The first author was supported by 
a University of Minnesota Undergraduate Research Scholarship.}
\date{November 2020}

\begin{abstract}
We compute explicitly the submodule structure of the Young modules for symmetric groups $S_n$ over fields of characteristic 2, when $n\le 7$. We use this information to compute the submodule structure of indecomposable projectives for the corresponding Schur algebras when $n\le 5$, and we give give partial information when $n=6,7$, including the Gabriel quiver and the structure of the Weyl modules. We resolve all Morita equivalences between blocks of these algebras.
\end{abstract}

\maketitle

\section{Introduction}

The Young modules $Y^\lambda$ play a fundamental role in the representation theory of the symmetric groups $S_n$. They are the indecomposable summands of the Young permutation modules $M^\lambda$, which are the permutation modules on the cosets of the Young subgroups $S_\lambda$. The isomorphism types of the $Y^\lambda$, as well as the $M^\lambda$ and $S_\lambda$, are indexed by partitions $\lambda$ of $n$. Over fields of characteristic 0 the Young modules coincide with the Specht modules, and form a complete set of irreducible modules. The situation is different over fields of positive characteristic $p$. Now the Young modules are no longer irreducible, in general, and they provide a transition between the irreducible modules, the Specht modules, and the larger projective indecomposable modules. Being typically more tractable than the Specht modules or the irreducible modules, they are useful in constructing larger representations. They also play a key role in the study of the Schur algebras, via the Schur-Weyl correspondence.

Explicit computations of the structure of Young modules are only known for symmetric groups of small degree, or for partitions with few parts, or of a special kind. In this work we extend such computations to all Young modules as far as the symmetric group $S_7$ in characteristic 2.
We also describe explicitly the structure of the Schur algebras $S_K(n,n)$ with $n\le 5$ and $K$ of characteristic 2, giving partial information when $n=6,7$. 
Our descriptions are by means of diagrams. Included in our description of the Young modules are the filtrations by Specht modules that correspond to Weyl filtrations of the projective modules for the Schur algebra. Our methods depend heavily on computer calculation.

Some of our descriptions of Young modules $Y^\lambda$ already appear in the literature, and we include the known cases here for completeness, so as to give a comprehensive picture. Descriptions of Young modules that are already known can be found as follows. When $\lambda=[n]$ and $\lambda=[n-1,1]$ the Young modules are well-known and their structure follows directly from \cite{Jam1}. When $\lambda=[n-2,2]$ and $\lambda=[n-2,1^2]$ they are described in \cite{BK} and \cite{MO}, for instance. In \cite{KMT} the structure is given when $\lambda=[n-3,3]$. When $\lambda$ is $p$-regular the Young module $Y^\lambda$ is projective, and when $p=2$ the projective modules are known for $S_n$ when $n\le 5$, see \cite{Ben1}. 
The radical layers of the indecomposable projectives are known when $n=6$ and are shown in \cite{Ben1}, but a full description of the projectives when $n=6$, in terms of diagrams showing non-split sections, was not given there.
We note that several of the references just mentioned provide good examples of how precise descriptions of Young modules can be used to resolve other questions in representation theory.


From our calculations, we can make some observations about the structure of the blocks that appear. We recall that each block of representations of a symmetric group is assigned a weight, that is the number of rim $p$-hooks that must be removed from any partition parametrizing a representation such as $Y^\lambda$ so as to leave a $p$-core.
The representations we consider all lie in blocks of weight 0, 1, 2 or 3. 
The blocks of weight 0 are simple projective modules (blocks of defect 0) and the blocks of weight 1 were shown to be Morita equivalent by Scopes \cite{Sco}. From our calculations with the remaining blocks of weights 2 and 3 we obtain the following.

\begin{corollary}
The principal block of $\FF_2S_5$ and the non-principal block of $\FF_2S_7$ of weight 2 are Morita equivalent, as are the corresponding blocks of the Schur algebras $S_{\FF_2}(5,5)$ and $S_{\FF_2}(7,7)$.
\end{corollary}

By contrast, $\FF_2S_4$ (which consists of a single block, of weight 2) is not Morita equivalent to the principal block of $\FF_2S_5$ or the non-principal block of $\FF_2S_7$. At the same time, the Schur algebra $S_{\FF_2}(4,4)$ is not Morita equivalent to  the corresponding blocks of the Schur algebras $S_{\FF_2}(5,5)$ and $S_{\FF_2}(7,7)$. Furthermore, the principal blocks of $\FF_2S_6$ and $\FF_2S_7$ (of weight 3), as well as the corresponding blocks of $S_{\FF_2}(6,6)$ and $S_{\FF_2}(7,7)$ are not Morita equivalent. These inequivalences can be seen because the Cartan matrices are different, and this information is readily available in the literature.

\begin{proof}
The proofs depend on calculations done in later sections. 
The diagrams for the projective modules in the blocks mentioned of $\FF_2S_5$ and $\FF_2S_7$ have the same structure, with different simple modules appearing. This shows that these group algebra blocks are Morita equivalent. The fact that the diagrams for all the Young modules in these blocks have the same structure shows that the blocks of the Schur algebras are Morita equivalent.
\end{proof}

For standard notation and facts about group representations, such as terminology for radical and socle layers etc, we refer to \cite{Web1}. For background on the representation theory of the symmetric groups and Schur algebras, see \cite{Jam1}, \cite{JK}, \cite{Sag}, \cite{Gre1} and \cite{Mar}. For an overview of this theory, the exposition \cite{Erd1} is very useful.

\section{Young modules and Schur algebras}
In this section we summarize the part of  the theory of representations of the symmetric groups necessary for our calculations, thereby establishing our notation. 


For each partition $\lambda=(\lambda_1,\lambda_2,\ldots,\lambda_t)$ of $n$ the \textit{Young subgroup} $S_\lambda$ of $S_n$ is defined as 
$$
S_\lambda=S_{\lambda_1}\times S_{\lambda_2}\times \cdots\times S_{\lambda_t}.
$$
and the corresponding permutation $RS_n$-module on the left cosets of $S_\lambda$ is
$$M^\lambda = R[S_n/S_\lambda] =RS_n\otimes_{RS_\lambda} R,$$
which we refer to as a \textit{Young permutation module}. 
It has the \textit{Specht module} $S^\lambda$ as a submodule.

When $R$ is a field of characteristic zero, these Specht modules are irreducible, and they provide a complete set of non-isomorphic irreducible representations of $S_n$. 
When $R$ is a field of characteristic $p>0$ and $\lambda$ is \textit{$p$-regular}, meaning that $\lambda$ has at most $p-1$ parts of each size, James showed that $S^\lambda$ has a unique simple quotient $D^\lambda$, and these $D^\lambda$ form a complete list of isomorphism types of irreducible modules \cite{Jam1}. 

The indecomposable direct summands of the Young permuatation modules $M^\lambda$ were considered by James in \cite{Jam2}, and a little later Grabmeier~\cite{Gra} and Green~\cite{Gre2} referred to them as \textit{Young modules}. 
James showed (at least over an infinite field, but this restriction was removed by Green \cite{Gre2}) that
\begin{enumerate}
    \item the Young modules $Y^\lambda$ are parametrized up to isomorphism by the partitions $\lambda$ of $n$, in such a way that 
    \item the summands of $M^\lambda$ all have the form $Y^\mu$ for some $\mu\ge\lambda$ in the dominance ordering, and that 
    \item $Y^\lambda$ occurs as a direct summand of $M^\lambda$ with multiplicity 1. 
\end{enumerate}
\noindent These conditions imply
\begin{enumerate}
    \item[(4)] The $Y^\lambda$ are self-dual.
\end{enumerate}
We use these properties in constructing the $Y^\lambda$ and in analyzing their structure. Unlike the Specht modules, we are not able to identify $Y^\lambda$ as a specific subset of $M^\lambda$, and so in the non-projective cases we construct $Y^\lambda$ by decomposing $M^\lambda$ into direct summands and identifying the unique new summand that has not appeared in earlier $M^\mu$. The self-duality is a strong constraint that is very useful computationally.

The Young modules have a close connection with the \textit{Schur algebras}, for which we refer to \cite{Gre1} and \cite{Mar} as background.  
We are interested in the submodule structure of the indecomposable projective modules of Schur algebras, and for this it is sufficient to work with basic algebras, Morita equivalent to the Schur algebras. By abuse of a usual notation for Schur algebras, we define $S_R(n,r) =\End_{S_r}(\bigoplus_\lambda Y^\lambda)$ when $n\ge r$. For each $r$, the algebras defined in this way are all the same provided $n\ge r$, and their common value is  $S_R(r,r)$. Conventionally, $r$ is the subscript of the symmetric group in this context, but because we want to use $n$ as this symbol we prefer to write the algebra associated to $S_n$ as $S_R(n,n)$. This algebra is Morita equivalent to the usual Schur algebra with these parameters. 

Our definition of Schur algebras is unusual in another respect, in that we allow $R$ to be any field (or complete local ring). Often the Schur algebras are only defined over infinite fields. When $R$ is a prime field, it is already a splitting field for $S_R(n,n)$, by work of Green \cite{Gre2}, so that in determining the submodule structure of the projectives we might as well work over a prime field.

It was shown by Donkin that the Young modules have Specht filtrations, that is, chains of submodules whose factors are Specht modules (see, for instance, \cite{Mar}). Such filtrations may be obtained as the images under the Schur functor of Weyl filtrations of the indecomposable projective modules for $S_R(n,n)$. We indicate the Specht filtrations that arise in this way in our results on the Young modules.



\section{Diagrammatic descriptions of modules}
The diagrams we use to describe the structure of modules have a long history, and within the context of representations of algebras go back at least to the 1960s, if not before. We will use diagrams in the sense of Benson and Carlson~\cite{BC} and review what it means for a diagram to describe a module. An earlier approach to this kind of theory, applicable in slightly more restrictive circumstances, was given by Alperin~\cite{Alp}. A brief introduction to the kind of diagrams we shall use can be found on pages 1556--7 of \cite{Mor}.

Briefly, our module diagrams are directed graphs $X$ without loops or multiple edges, together with a labeling of the vertices by a fixed set of irreducible modules. In drawing these diagrams, the edges always point down. Each edge is labeled by a non-zero Ext class so that if the edge $x\xrightarrow{\alpha}y $ has source $x$ labeled by the irreducible $S$ and the target $y$ is labeled by the irreducible $T$, then $\alpha\in \Ext^1(S,T)$. 

To say what it means for such a diagram $X$ to describe a certain module $M$, we consider subsets of the set of vertices of $X$ that are closed under following along directed edges. For each such `arrow closed' subset $U$ of vertices there is a submodule $M_U$ associated to it, with the property if $U_1\supseteq U_2$ then $M_{U_1}\supseteq M_{U_2}$, and so that the composition factors of $M_{U_1}/M_{U_2}$ are the labels of the vertices in $U_1$ that are not in $U_2$. We require that $M_X=M$ and $M_\emptyset =0$, from which it follows that the labels of the vertices of $X$ are exactly the composition factors of $M$, taken with multiplicity and up to isomorphism. If $U_1\supseteq U_2$ differ by two vertices joined by an edge $x\xrightarrow{\alpha}y $ with $x$ labeled by $S$ and $y$ labeled by $T$ then the extension class of the short exact sequence
$$
0\to M_{U_2\cup\{T\}}/M_{U_2} \to  M_{U_1}/M_{U_2} \to  M_{U_1}/M_{U_2\cup\{T\}} \to 0
$$
is $\alpha$. If there is no edge between $x$ and $y$, the above extension must split. In our diagrams we do not label edges when $\dim\Ext(S,T)=1$, because over $\FF_2$ there is only one non-zero extension. Thus the existence of an edge simply means a non-split extension. 

The only case of an Ext group of dimension larger than 1 that we shall encounter is with the group $S_7$, where $\dim\Ext(D^{[5,2]},D^{[5,2]})=2$, and in this case we give the edges labels. We point out that for the Young module $Y^{[3,1^4]}$ for $S_7$, the situation arises where three copies of $D^{[5,2]}$ appear underneath a single copy of $D^{[5,2]}$, each edge labeled by a different non-zero element of the Ext group. It seems not to be possible in this case to have a diagram where the labels on edges emanating from a single vertex are linearly independent.



Part of the information conveyed by a diagram for a module $M$ is a description of its radical and socle layers, $\Rad^iM / \Rad^{i+r}M$ and $\Soc^{i+r}M/\Soc^{i}M$. To obtain this description, note that the radical of $M$ has a diagram obtained by removing all vertices that are not targets of arrows, and the socle of $M$ has as a diagram the vertices that are not sources of arrows (a discrete set of points). The powers of the radical and socle are obtained by iterating these combinatorial procedures.

As a special case of this, a module is semisimple if and only if it has a diagram with no edges. More generally, a module with a diagram that is a disjoint union of subdiagrams is the direct sum of the submodules with those diagrams.

At the other extreme, a diagram where the vertices are arranged in a single vertical line of edges represents a \textit{uniserial module}, namely, one with a unique composition series or, equivalently, where the submodules are linearly ordered by inclusion. This is so because a module is uniserial if and only if all its radical layers $\Rad^i M / \Rad^{i+1} M$ are simple, and this happens if and only if the diagram has a single line of edges. When drawing uniserial diagrams we often omit the directed arrows between the composition factors where this does not cause confusion. 



\section{Computational methods}
Our methods are heavily computational, using routines developed by the authors and others in GAP, and available from from Webb's internet site \cite{Web2}. At a basic level these routines provide functionality to compute with submodules and quotient modules of given representations. Further routines allow more elaborate operations. We highlight some of the key approaches.  Specific further detail is given in notes for each module we consider.

Our initial step is to construct the Young modules. These are summands of permutation modules and, when they are small enough, can be constructed by decomposing the permutation module into direct summands. The basic algorithm used to do this depends on Fitting's lemma, and finds a random endomorphism of the module that is not nilpotent and not an isomorphism. Raising this endomorphism to a high enough power, its kernel and image provide a direct sum decomposition of the module. This general approach is demanding on computational resources, but succeeds with modules whose dimensions are in the low hundreds.

With larger modules we use a different approach to direct sum decomposition that applies when the isomorphism types of some summands are already known. It works with the permutation modules $M^\lambda$ when the $\lambda$ are taken in dominance order, because each $M^\lambda$ has a unique new direct summand $Y^\lambda$ and all the other direct summands have already been constructed. The largest module $M^\lambda$ that we decomposed in this way has dimension 840. This approach may be new in the computational context, although the idea is very simple. Given a module and a candidate summand, random homomorphisms between these modules in both directions are tested until they happen to be found so that the composite from the summand to the bigger module, back to the summand is an isomorphism. These maps are then split mono and split epi, and express the candidate summand as a direct summand of the larger module. As a special case, this approach also provides a way to test for isomorphism of two modules, and find an isomorphism when they are isomorphic.

Many of the larger projective Young modules are constructed as summands of tensor products of suitably chosen smaller modules. We can test whether a summand of such a tensor product is projective by a routine that restricts to a Sylow $p$-subgroup and checks the rank of the sum of the group elements. If the module then has a unique simple quotient, it is the indecomposable projective associated to that module.

Once the Young modules have been constructed, a first approach to analyzing their structure is to compute the factors that appear in the Zassenhaus Butterfly Lemma, associated to two chains of submodules. We use the radical and socle series as input, as computed by the Meat-axe \cite{Par}. This gives a list of composition factors of the module, positioned according to the radical and socle layers in which they appear. The approach immediately identifies uniserial modules, and is a start for more complicated modules.

After this, a very useful routine allows us to remove from the bottom of a module all homomorphic images of another module, meaning that we quotient out the sum of the images of all homomorphisms from one module to another. A dual routine computes the intersection of the kernels of all homomorphisms from one module to another, allowing us to remove composition factors from the top of a module in a controlled way. These routines allow us to remove parts of a module with some precision so that we can examine what remains. By judiciously removing submodules and factor modules in this way, we can test for split extensions by direct sum decomposition and by examining radical and socle layers.

Certain kinds of modules appear repeatedly in our investigation, and are easily identified. The structure of uniserial modules is identified as already noted. It is often useful to work with the \textit{heart} of a module, which may be defined to be the quotient of its radical by its socle (in situations where the socle is contained in the radical). This is calculated by removing the socle and the top as just noted, and the structure of modules whose hearts are the direct sum of two uniserial modules can then be immediately verified. Such modules are called \textit{biserial}. We will use the term \textit{string module} to indicate a module with a diagram that is combinatorially a subdivision of a single line, folded so that the edges may go up or down. Such modules are identified by removing composition factors from the socle and the top in a judicious way and then showing that splitting occurs. Details are given of the precise procedure when this occurs.

Other computational strategies proceeded on an ad hoc basis, and are detailed in notes after each diagram.

\section{The module structure of Young modules in characteristic 2}


At the start of the tables for each symmetric group we give the table of multiplicities of the Young modules as summands of the Young permutation modules. These multiplicities are known as the \textit{$p$-Kostka numbers} (here $p=2$). They are already known, because they may be computed from the decomposition numbers for the Schur algebras given in \cite{Jam3} and \cite{Gra}. We present the table for the convenience of the reader, because the multiplicities are important in our construction of the Young modules, and we do not know of a reference where they are stated explicitly.




The factors in a Specht module filtration of the Young modules are indicated either by enclosing the factors for a single Specht module in a box, or by joining those factors by thick lines. We have chosen to use more than one way to indicate the Specht factors because of the difficulty of incorporating this information in diagrams that may already be quite complicated. Using thick lines is often more straightforward, but it is not always possible. When a Specht module is simple, there are no lines to thicken, and in these cases we enclose the Specht module in a box. The Specht modules $S^{[5,1^2]}$ and $S^{[3,1^4]}$ decompose as the direct sum of two simple modules, and so we enclose those modules in a box, again because there is no edge to thicken. In each situation we have chosen thick lines or boxes to produce the diagram that conveys the information most clearly, in our view. 

The Specht module factors in our Specht filtrations are computed by applying the Schur functor to the Weyl modules of the Schur algebra, which gives their composition factors and also the order in which they must appear, using the quasi-hereditary property of the Schur algebra. Under this construction, the Schur module $S^\lambda$ appears as the bottom factor in the Specht filtration of $Y^\lambda$, so that our diagrams include the structure of the Specht modules.

\subsection{Young modules for $S_1$}

\noindent Young module summand multiplicities:

\begin{center}
{\renewcommand{\arraystretch}{1.2}
\begin{tabular}{ |c|c||c| }
\hline
\multicolumn{2}{|c||}{Table of}&Permutation module\\
\cline{3-3}
\multicolumn{2}{|c||}{$2$-Kostka numbers}&$[1]$\\
\hline\hline
Young module & $[1]$ & 1 \\
\hline
\end{tabular}}
\end{center}
\noindent Young module structure:

\begin{center}
\begin{tabular}{ |c||c| }
\hline
$\lambda$&$[1]$\\
\hline\hline
$Y^\lambda$&
\begin{tikzpicture}[baseline=0pt]
\node[inner sep=0.15] (1) at  (0,0.125)  {$D^{[1]}$};
\node[draw=gray, fit = (1), inner sep=1] {};
\node () at (0,.275) {};
\node () at (0,-.025) {};
\end{tikzpicture}

\\
\hline
\end{tabular}
\end{center}

\subsection{Young modules for $S_2$}

\noindent Young module summand multiplicities:
\begin{center}
{\renewcommand{\arraystretch}{1.2}
\begin{tabular}{ |c|c||c|c| }
\hline
\multicolumn{2}{|c||}{Table of}&\multicolumn{2}{|c|}{Permutation module}\\
\cline{3-4}
\multicolumn{2}{|c||}{$2$-Kostka numbers}&$[2]$&$[1^2]$\\
\hline\hline
\multirow{2}{3em}{Young module} & $[2]$ & 1 & 0 \\
&$[1^2]$ & 0 & 1 \\
\hline
\end{tabular}}
\end{center}
\noindent Young module structure:

\begin{center}
\begin{tabular}{ |c||c|c| }
\hline
$\lambda$&$[2]$&$[1^2]$\\
\hline\hline
$Y^\lambda$&
\begin{tikzpicture}[baseline=0pt]
\node[inner sep=0.15] (1) at  (0,.175)  {$D^{[2]}$};
\node[draw=gray, fit = (1), inner sep=1] {};
\end{tikzpicture}
&
\begin{tikzpicture}[baseline=0pt]
\node[inner sep=0.15] (1) at  (0,.45)  {$D^{[2]}$};
\node[inner sep=0.15] (2) at  (0,-.05)   {$D^{[2]}$};
\node[draw=gray, fit = (1), inner sep=.5] {};
\node[draw=gray, fit = (2), inner sep=.5] {};
\node () at (0,.6) {};
\node () at (0,-.2) {};
\end{tikzpicture}
\\
\hline
\end{tabular}
\end{center}

\subsection{Young modules for $S_3$}

\noindent Young module summand multiplicities:
\begin{center}
{\renewcommand{\arraystretch}{1.2}
\begin{tabular}{ |c|c||c|c|c| }
\hline
\multicolumn{2}{|c||}{Table of}&\multicolumn{3}{|c|}{Permutation module}\\
\cline{3-5}
\multicolumn{2}{|c||}{$2$-Kostka numbers}&$[3]$&$[2,1]$&$[1^2]$\\
\hline\hline
\multirow{3}{3em}{Young module} & $[3]$ & 1 & 1 & 0 \\
&$[2,1]$ & 0 & 1 & 2 \\
&$[1^3]$ & 0 & 0 & 1 \\
\hline
\end{tabular}}
\end{center}

\noindent Young module structure:
\begin{center}
\begin{tabular}{ |c||c|c|c| }
\hline
$\lambda$&$[3]$&$[2,1]$&$[1^3]$\\
\hline\hline
$Y^\lambda$&
\begin{tikzpicture}[baseline=0pt]
\node[inner sep=0.5] (1) at  (0,.175)  {$D^{[3]}$};
\node[draw=gray, fit = (1), inner sep=1] {};
\end{tikzpicture}
&
\begin{tikzpicture}[baseline=0pt]
\node[inner sep=0.5] (1) at  (0,.175)  {$D^{[2,1]}$};
\node[draw=gray, fit = (1), inner sep=1] {};
\end{tikzpicture}&
\begin{tikzpicture}[baseline=0pt]
\node[inner sep=0.15] (1) at  (0,.4)  {$D^{[3]}$};
\node[inner sep=0.15] (2) at  (0,-.1)   {$D^{[3]}$};
\node[draw=gray, fit = (1), inner sep=.5] {};
\node[draw=gray, fit = (2), inner sep=.5] {};
\node () at (0,.55) {};
\node () at (0,-.25) {};
\end{tikzpicture}

\\
\hline
\end{tabular}
\end{center}

\subsection{Young modules for $S_4$}

\noindent Young module summand multiplicities:
\begin{center}
{\renewcommand{\arraystretch}{1.2}
\begin{tabular}{ |c|c||c|c|c|c|c| }
\hline
\multicolumn{2}{|c||}{Table of}&\multicolumn{5}{|c|}{Permutation module}\\
\cline{3-7}
\multicolumn{2}{|c||}{$2$-Kostka numbers}&$[4]$&$[3,1]$&$[2^2]$&$[2,1^2]$&$[1^4]$\\
\hline\hline
\multirow{5}{3em}{Young module} & $[4]$ & 1 & 0 & 0 & 0 & 0\\
&$[3,1]$ & 0 & 1 & 0 & 1 & 0\\
&$[2^2]$ & 0 & 0 & 1 & 0 & 0 \\
&$[2,1^2]$ & 0 & 0 & 0 & 1 & 2 \\
&$[1^4]$ & 0 & 0 & 0 & 0 & 1 \\
\hline
\end{tabular}}
\end{center}

\noindent Young module structure:

\begin{center}
\begin{tabular}{ |c||c|c|c|c|c| }
\hline
$\lambda$&$[4]$&$[3,1]$&$[2^2]$&$[2,1^2]$& $[1^4]$\\
\hline\hline
$Y^\lambda$&
\begin{tikzpicture}[baseline=0pt]
\node[inner sep=0.5] (1) at  (0,.175)  {$D^{[4]}$};
\node[draw=gray, fit = (1), inner sep=1] {};
\end{tikzpicture}&

\begin{tikzpicture}[baseline=0pt]
\node[inner ysep=0.15] (1) at  (0,.8)  {$D^{[4]}$};
\node[inner sep=0.15] (2) at  (0,0.2)   {$D^{[3,1]}$};
\node[inner sep=0.15] (3) at  (0,-.4)  {$D^{[4]}$};
\node[draw=gray, fit = (1), inner sep=.5] {};
\node[draw=gray, fit = (2) (3), inner sep=.5] {};
\end{tikzpicture}

&
\begin{tikzpicture}[baseline=0pt]
\node[inner xsep=0.15, inner ysep=1] (1) at (0,.75) {${D^{[4]}}$};
\node[inner xsep=0.15, inner ysep=1] (2) at (.7,-.25) {${D^{[3,1]}}$};
\node[inner xsep=0.15, inner ysep=1] (3) at (1.4,.75) {${D^{[3,1]}}$};
\node[inner xsep=0.15, inner ysep=1] (4) at (2.1,-.25) {${D^{[4]}}$};

\begin{pgfonlayer}{bg}
\draw[ultra thick] (1) -- (2);
\draw (2) -- (3);
\draw[ultra thick] (3) -- (4);
\end{pgfonlayer}
\end{tikzpicture}
&

\begin{tikzpicture}[baseline=0pt]
\node[inner xsep=0.15, inner ysep=1] (1) at  (0,1.45)  {$D^{[3,1]}$};
\node[inner xsep=0.15, inner ysep=1] (l1) at  (-.75,.65)   {$D^{[4]}$};
\node[inner xsep=0.15, inner ysep=1, draw=gray] (r) at  (.75,0.25)   {$D^{[3,1]}$};
\node[inner xsep=0.15, inner ysep=1] (l2) at  (-.75,-.15)   {$D^{[4]}$};
\node[inner xsep=0.15, inner ysep=1] (bot) at  (0,-.95)   {$D^{[3,1]}$};
\node () at (0,1.6) {};
\node () at (0,-1.1) {};

\begin{pgfonlayer}{bg}
\draw[ultra thick] (1) -- (l1);
\draw (l1) -- (l2);
\draw[ultra thick] (l2) -- (bot);
\draw (1) -- (r) -- (bot);
\end{pgfonlayer}
\end{tikzpicture}

&

\begin{tikzpicture}[baseline=0pt]
\node[inner sep=0.15] (top) at  (0,1.4)  {$D^{[4]}$};
\node[inner sep=0.15] (l1) at  (-.7,.7)   {$D^{[4]}$};
\node[inner sep=0.15] (r1) at  (.7,.7)   {$D^{[3,1]}$};
\node[inner sep=0.15] (l2) at  (-.7,0)   {$D^{[3,1]}$};
\node[inner sep=0.15] (r2) at  (.7,0)   {$D^{[4]}$};
\node[inner sep=0.15] (bot) at  (0,-.7)   {$D^{[4]}$};
\node[inner sep=0.15] at  (0,.25)   {$\oplus$};
\node[draw=gray, fit = (top), inner sep=1] {};
\node[draw=gray, fit = (l1) (l2), inner sep=1] {};
\node[draw=gray, fit = (bot), inner sep=1] {};
\node[draw=gray, fit = (r1) (r2), inner sep=1] {};
\node () at (0,1.15) {};
\node () at (0,-.65) {};
\end{tikzpicture}
\\
\hline
\end{tabular}
\end{center}

\textbf{Notes on $Y^{[2^2]}$.} We compute that the socle has the form $D^{[3,1]}\oplus D^{[4]}$. Quotienting out the socle $D^{[4]}$  leaves an indecomposable module, while quotienting out instead $D^{[3,1]}$ leaves two uniserial summands. The structure of these summands implies the diagram shown.

\subsection{Young modules for $S_5$}

\noindent Young module summand multiplicities:
\begin{center}
{\renewcommand{\arraystretch}{1.2}
\begin{tabular}{ |c|c||c|c|c|c|c|c|c| }
\hline
\multicolumn{2}{|c||}{Table of}&\multicolumn{7}{|c|}{Permutation module}\\
\cline{3-9}
\multicolumn{2}{|c||}{$2$-Kostka numbers}&$[5]$&$[4,1]$&$[3,2]$&$[3,1^2]$&
$[2^2,1]$&$[2,1^3]$&
$[1^5]$\\
\hline\hline
\multirow{7}{3em}{Young module} & $[5]$ & 1 & 1 & 0 & 0 & 0 & 0 & 0\\
&$[4,1]$ & 0 & 1 & 1 & 2 & 2 & 2 & 0\\
&$[3,2]$ & 0 & 0 & 1 & 0 & 1 & 0 & 0\\
&$[3,1^2]$ & 0 & 0 & 0 & 1 & 0 & 1 & 0\\
&$[2^2,1]$ & 0 & 0 & 0 & 0 & 1 & 2 & 4\\
&$[2,1^3]$ & 0 & 0 & 0 & 0 & 0 & 1 & 4\\
&$[1^5]$ & 0 & 0 & 0 & 0 & 0 & 0 & 1\\
\hline
\end{tabular}}
\end{center}
\noindent Young module structure:

\begin{center}
\begin{tabular}{ |c||c|c|c|c|c|c|c| }
\hline
$\lambda$&$[5]$&$[4,1]$&$[3,2]$&$[3,1^2]$\\
\hline\hline
$Y^\lambda$&
\begin{tikzpicture}[baseline=0pt]
\node[inner sep=0.5] (1) at  (0,.175)  {$D^{[5]}$};
\node[draw=gray, fit = (1), inner sep=1] {};
\end{tikzpicture}
&
\begin{tikzpicture}[baseline=0pt]
\node[inner sep=0.5] (1) at  (0,.175)  {$D^{[4,1]}$};
\node[draw=gray, fit = (1), inner sep=1] {};
\end{tikzpicture}&
\begin{tikzpicture}[baseline=0pt]
\node[inner ysep=0.15] (1) at  (0,.5)  {$D^{[5]}$};
\node[inner sep=0.15] (2) at  (0,0)   {$D^{[3,2]}$};
\node[inner sep=0.15] (3) at  (0,-.5)  {$D^{[5]}$};
\node[draw=gray, fit = (1), inner sep=.5] {};
\node[draw=gray, fit = (2) (3), inner sep=.5] {};
\end{tikzpicture}

&
\begin{tikzpicture}[baseline=0pt]
\node[inner ysep=0.15] (1) at  (0,1.25)  {$D^{[5]}$};
\node[inner sep=0.15] (2) at  (0,.75)   {$D^{[3,2]}$};
\node[inner sep=0.15] (3) at  (0,.25)  {$D^{[5]}$};
\node[inner ysep=0.15] (4) at  (0,-.25)  {$D^{[5]}$};
\node[inner ysep=0.15] (5) at  (0,-.75)   {$D^{[3,2]}$};
\node[inner ysep=0.15] (6) at  (0,-1.25)  {$D^{[5]}$};
\node[draw=gray, fit = (1), inner sep=.25] {};
\node[draw=gray, fit = (2) (3), inner sep=.5] {};
\node[draw=gray, fit = (4) (6), inner sep=.5] {};
\node () at (0,1.4) {};
\node () at (0,-1.4) {};
\end{tikzpicture}

\\
\hline
\end{tabular}
\end{center}

\begin{center}
\begin{tabular}{ |c||c|c|c| }
\hline
$\lambda$&$[2^2,1]$&$[2,1^3]$&$[1^5]$\\
\hline\hline
$Y^\lambda$&
\begin{tikzpicture}[baseline=0pt]
\node[inner sep=0.15] (1) at  (0,1.5)  {$D^{[3,2]}$};
\node[inner sep=0.15] (2) at  (0,1)   {$D^{[5]}$};
\node[inner ysep=0.15] (3) at  (0,.5)  {$D^{[5]}$};
\node[inner ysep=0.15] (4) at  (0,0)   {$D^{[3,2]}$};
\node[inner ysep=0.15] (5) at  (0,-.5)   {$D^{[5]}$};
\node[inner sep=0.15] (6) at  (0,-1)  {$D^{[5]}$};
\node[inner sep=0.15] (7) at  (0,-1.5)   {$D^{[3,2]}$};
\node[draw=gray, fit = (1) (2), inner sep=.5] {};
\node[draw=gray, fit = (3) (5), inner sep=.5] {};
\node[draw=gray, fit = (6) (7), inner sep=.5] {};
\node () at (0,1.65) {};
\node () at (0,-1.65) {};
\end{tikzpicture}

&

\begin{tikzpicture}[baseline=0pt]
\node[inner sep=0.15] (1) at  (0,.25)  {$D^{[4,1]}$};
\node[inner sep=0.15] (2) at  (0,-.25)   {$D^{[4,1]}$};
\node[draw=gray, fit = (1), inner sep=.5] {};
\node[draw=gray, fit = (2), inner sep=.5] {};
\node () at (0,.4) {};
\node () at (0,-.4) {};
\end{tikzpicture}

&
\begin{tikzpicture}[baseline=0pt]

\node[inner sep=0.15] (top) at  (0,1.5)   {$D^{[5]}$};
\node[inner sep=0.15] (bot) at  (0,-1.5)   {$D^{[5]}$};
\node[inner sep=0.15] (plus) at  (0,0)   {$\oplus$};
\node[inner ysep=.15] (l1) at  (-1,1)  {$D^{[5]}$};
\node[inner sep=0.15] (l2) at  (-1,.5)   {$D^{[3,2]}$};
\node[inner sep=0.15] (l3) at  (-1,0)   {$D^{[5]}$};
\node[inner sep=0.15] (l4) at  (-1,-.5)   {$D^{[5]}$};
\node[inner sep=0.15] (l5) at  (-1,-1)  {$D^{[3,2]}$};
\node[inner sep=0.15] (r1) at  (1,1)  {$D^{[3,2]}$};
\node[inner sep=0.15] (r2) at  (1,.5)   {$D^{[5]}$};
\node[inner sep=0.15] (r3) at  (1,0)   {$D^{[5]}$};
\node[inner sep=0.15] (r4) at  (1,-.5)   {$D^{[3,2]}$};
\node[inner ysep=.15] (r5) at  (1,-1)  {$D^{[5]}$};
\node[draw=gray, fit = (top), inner sep=1] {};
\node[draw=gray, fit = (bot), inner sep=1] {};
\node[draw=gray, fit = (l1) (l3), inner sep=.5] {};
\node[draw=gray, fit = (r1) (r2), inner sep=.5] {};
\node[draw=gray, fit = (l4) (l5), inner sep=.5] {};
\node[draw=gray, fit = (r3) (r5), inner sep=.5] {};
\node () at (0,1.65) {};
\node () at (0,-1.65) {};
\end{tikzpicture}

\\
\hline
\end{tabular}
\end{center}

\subsection{Young modules for $S_6$}

\noindent Young module summand multiplicities:
\begin{center}
{\renewcommand{\arraystretch}{1.2}
\scalebox{0.75}{
\begin{tabular}{ |c|c||c|c|c|c|c|c|c|c|c|c|c| }
\hline
\multicolumn{2}{|c||}{Table of}&\multicolumn{11}{|c|}{Permutation module}\\
\cline{3-13}
\multicolumn{2}{|c||}{$2$-Kostka numbers}&$[6]$&$[5,1]$&$[4,2]$&$[4,1^2]$&
$[3^2]$&$[3,2,1]$&$[3,1^3]$&$[2^3]$&$[2^2,1^2]$&
$[2,1^4]$&$[1^6]$\\
\hline\hline
\multirow{11}{3em}{Young module} & $[6]$ & 1 & 0 & 1 & 0 & 0 & 0 & 0 & 0 & 0 & 0 & 0\\
&$[5,1]$ & 0 & 1 & 0 & 1 & 0 & 0 & 0 & 0 & 0 & 0 & 0\\
&$[4,2]$ & 0 & 0 & 1 & 0 & 0 & 0 & 0 & 2 & 0 & 0 & 0\\
&$[4,1^2]$ & 0 & 0 & 0 & 1 & 0 & 1 & 2 & 0 & 2 & 2 & 0\\
&$[3^2]$ & 0 & 0 & 0 & 0 & 1 & 1 & 0 & 0 & 1 & 0 & 0\\
&$[3,2,1]$ & 0 & 0 & 0 & 0 & 0 & 1 & 2 & 2 & 4 & 8 & 16\\
&$[3,1^3]$ & 0 & 0 & 0 & 0 & 0 & 0 & 1 & 0 & 0 & 1 & 0\\
&$[2^3]$ & 0 & 0 & 0 & 0 & 0 & 0 & 0 & 1 & 0 & 0 & 0\\
&$[2^2,1^2]$ & 0 & 0 & 0 & 0 & 0 & 0 & 0 & 0 & 1 & 2 & 4\\
&$[2,1^4]$ & 0 & 0 & 0 & 0 & 0 & 0 & 0 & 0 & 0 & 1 & 4\\
&$[1^6]$ & 0 & 0 & 0 & 0 & 0 & 0 & 0 & 0 & 0 & 0 & 1\\
\hline
\end{tabular}}}
\end{center}
\noindent Young module structure:

\begin{center}
\begin{tabular}{ |c||c|c|c|c|c| }
\hline
$\lambda$&$[6]$&$[5,1]$&$[4,2]$&$[4,1^2]$&$[3^2]$\\
\hline\hline
$Y^\lambda$&
\begin{tikzpicture}[baseline=0pt]
\node[inner sep=0.5] (1) at  (0,.175)  {$D^{[6]}$};
\node[draw=gray, fit = (1), inner sep=1] {};
\end{tikzpicture} &
\begin{tikzpicture}[baseline=0pt]

\node[inner ysep=0.15] (1) at  (0,.5)   {$D^{[6]}$};
\node[inner sep=0.15] (2) at  (0,0)   {$D^{[5,1]}$};
\node[inner sep=0.15] (3) at  (0,-.5)  {$D^{[6]}$};
\node[draw=gray, fit = (1), inner sep=.5] {};
\node[draw=gray, fit = (2) (3), inner sep=.5] {};

\end{tikzpicture}

&
\begin{tikzpicture}[baseline=0pt]

\node[inner sep=0.15] (1) at  (0,1)   {$D^{[5,1]}$};
\node[inner sep=0.15] (2) at  (0,.5)   {$D^{[6]}$};
\node[inner sep=0.15] (3) at  (0,0)  {$D^{[4,2]}$};
\node[inner sep=0.15] (4) at  (0,-.5)   {$D^{[6]}$};
\node[inner sep=0.15] (5) at  (0,-1)   {$D^{[5,1]}$};
\node[draw=gray, fit = (1) (2), inner sep=.5] {};
\node[draw=gray, fit = (3) (5), inner sep=.5] {};

\end{tikzpicture}

&
\begin{tikzpicture}[baseline=0pt]

\node[inner sep=0.15] (1) at  (0,2)   {$D^{[5,1]}$};
\node[inner sep=0.15] (2) at  (0,1.5)   {$D^{[6]}$};
\node[inner sep=0.15] (3) at  (0,1)  {$D^{[4,2]}$};
\node[inner sep=0.15] (4) at  (0,.5)   {$D^{[6]}$};
\node[inner sep=0.15] (5) at  (0,0)   {$D^{[5,1]}$};
\node[inner sep=0.15] (6) at  (0,-.5)   {$D^{[6]}$};
\node[inner sep=0.15] (7) at  (0,-1)  {$D^{[4,2]}$};
\node[inner sep=0.15] (8) at  (0,-1.5)   {$D^{[6]}$};
\node[inner sep=0.15] (9) at  (0,-2)   {$D^{[5,1]}$};

\node[draw=gray, fit = (1) (2), inner sep=.5] {};
\node[draw=gray, fit = (3) (5), inner sep=.5] {};
\node[draw=gray, fit = (6) (9), inner sep=.5] {};

\end{tikzpicture}
&

\begin{tikzpicture}[baseline=0pt]

\node[inner xsep=0.15] (5) at  (1,-2)  {$D^{[6]}$};
\node[inner xsep=0.15] (41) at  (0,-1)   {$D^{[5,1]}$};
\node[inner xsep=0.15] (31) at  (0,0)   {$D^{[6]}$};
\node[inner xsep=0.15] (21) at  (0,1)   {$D^{[4,2]}$};
\node[draw=gray, inner sep=0.5] (1) at  (1,2)   {$D^{[6]}$};
\node[inner xsep=0.15] (22) at  (2,1)   {$D^{[5,1]}$};
\node[inner xsep=0.15] (32) at  (2,0)   {$D^{[6]}$};
\node[inner xsep=0.15] (42) at  (2,-1)   {$D^{[4,2]}$};
\node () at (1,2.15) {};

\begin{pgfonlayer}{bg}
\draw[ultra thick] (21)--(31)--(41);
\draw (41) --(5);
\draw[ultra thick] (5)--(42);
\draw (42)--(32);
\draw[ultra thick](32)--(22);
\draw (22)--(1)--(21) -- (42);
\end{pgfonlayer}
\end{tikzpicture}
\\
\hline
\end{tabular}
\end{center}

\textbf{Notes on $Y^{[3^2]}$.} The socle of the heart of $Y^{[3^2]}$ is $D^{[5,1]}\oplus D^{[4,2]}$. Quotienting out $D^{[4,2]}$ from the heart leaves two uniserial summands with composition factors shown, while the result of quotienting out successively $D^{[5,1]}$ and $D^{[6]}$ remains indecomposable. Now, the fact that the heart is self dual implies that the heart is a string module in the manner shown.

\begin{center}
\begin{tabular}{ |c||c|c|c| }
\hline
$\lambda$&$[3,2,1]$&$[3,1^3]$&$[2^3]$\\
\hline\hline
$Y^\lambda$ &
\begin{tikzpicture}[baseline=0pt]
\node[inner sep=0.5] (1) at  (0,.175)  {$D^{[3,2,1]}$};
\node[draw=gray, fit = (1), inner sep=1] {};
\end{tikzpicture}
&
\begin{tikzpicture}[baseline=0pt]
\node[inner xsep=0.15] (9) at  (0,-4)  {$D^{[6]}$};
\node[inner xsep=0.15] (81) at  (-1,-3)   {$D^{[5,1]}$};
\node[inner xsep=0.15] (71) at  (-1,-2)   {$D^{[6]}$};
\node[inner xsep=0.15] (61) at  (-1,-1)   {$D^{[4,2]}$};
\node[inner xsep=0.15] (51) at  (-1,0)   {$D^{[6]}$};
\node[inner xsep=0.15] (41) at  (-1,1)   {$D^{[5,1]}$};
\node[inner xsep=0.15] (31) at  (-1,2)   {$D^{[6]}$};
\node[inner xsep=0.15] (21) at  (-1,3)   {$D^{[4,2]}$};
\node[draw=gray, inner sep=0.5] (1) at  (0,4)  {$D^{[6]}$};
\node[inner xsep=0.15] (82) at  (1,-3)   {$D^{[4,2]}$};
\node[inner xsep=0.15] (72) at  (1,-2)   {$D^{[6]}$};
\node[inner xsep=0.15] (62) at  (1,-1)   {$D^{[5,1]}$};
\node[inner xsep=0.15] (52) at  (1,0)   {$D^{[6]}$};
\node[inner xsep=0.15] (42) at  (1,1)   {$D^{[4,2]}$};
\node[inner xsep=0.15] (32) at  (1,2)   {$D^{[6]}$};
\node[inner xsep=0.15] (22) at  (1,3)   {$D^{[5,1]}$};

\begin{pgfonlayer}{bg}
\draw (1) -- (22);
\draw[ultra thick] (22)-- (32);
\draw (32)-- (42);
\draw[ultra thick] (42)-- (52);
\draw (52)-- (62);
\draw[ultra thick] (62)-- (72) -- (82) -- (9);
\draw (9)-- (81);
\draw[ultra thick] (81)-- (71) -- (61) -- (51);
\draw (51)-- (41);
\draw[ultra thick] (41)-- (31) -- (21);
\draw (21)-- (1);
\draw (21)--(42);
\draw (31)--(52);
\draw (41)--(62);
\draw (51)--(72);
\draw (61)--(82);
\end{pgfonlayer}
\end{tikzpicture}

&

\begin{tikzpicture}[baseline=0pt]

\node[inner xsep=0.15] (9) at  (0,-4)  {$D^{[4,2]}$};
\node[inner xsep=0.15] (8) at  (0,-3)   {$D^{[6]}$};
\node[inner xsep=0.15] (7) at  (0,-2)   {$D^{[5,1]}$};
\node[inner xsep=0.15] (6) at  (0,-1)   {$D^{[6]}$};
\node[inner xsep=0.15] (5) at  (0,0)   {$D^{[4,2]}$};
\node[inner xsep=0.15] (4) at  (0,1)   {$D^{[6]}$};
\node[inner xsep=0.15] (3) at  (0,2)   {$D^{[5,1]}$};
\node[inner xsep=0.15] (2) at  (0,3)   {$D^{[6]}$};
\node[inner xsep=0.15] (1) at  (0,4)  {$D^{[4,2]}$};
\node[inner xsep=0.15] (22) at  (3,0)   {$D^{[4,2]}$};
\node[draw=gray, inner sep=0.5] (21) at  (3,4)   {$D^{[6]}$};
\node[inner xsep=0.15] (23) at  (3,-4)   {$D^{[6]}$};

\begin{pgfonlayer}{bg}
\draw[ultra thick] (1)--(2)--(3);
\draw (3)--(4);
\draw[ultra thick] (4)--(5)--(6)--(7);
\draw (7)--(8);
\draw[ultra thick](8)--(9);
\draw (21)--(22);
\draw[ultra thick](22)--(23);
\draw (1)--(22)--(9);
\draw (21)--(0.125,-3);
\draw (23)--(0.125,3);
\draw[white,fill=white] (2) circle [radius=0.3];
\draw[white,fill=white] (8) circle [radius=0.3];
\end{pgfonlayer}
\end{tikzpicture}
\\
\hline
\end{tabular}
\end{center}

\textbf{Notes on $Y^{[3,1^3]}$.} We first verify that each section of the form $\Rad^i Y^{[3,1^3]} / \Rad^{i+3} Y^{[3,1^3]}$ for  $1\le i\le 5$ is a string module as shown. These modules are all indecomposable. When $i=3$,  quotienting out $D^{[4,2]}$ and then $D^{[6]}$ yields an indecomposable module, while quotienting out instead $D^{[5,1]}$ causes the module to decompose as the direct sum of two uniserial modules. The cases for $i=1,5$ are similar. When $i=2,4$, it is more delicate to remove the $D^{[6]}$ modules correctly: where $i=2$ for example, we may work with the larger module $\Rad^2 Y^{[3,1^3]} / \Rad^6 Y^{[3,1^3]}$, and remove $D^{[5,1]}$ and then $D^{[6]}$.

Additionally, there is a homomorphism $Y^{[4,1^2]}\to Y^{[3,1^3]}$ whose image is uniserial of length 8, and whose cokernel is also uniserial of length 8, with composition factors as shown. Therefore the only edges between the two uniserial strands of the heart go from left to right as shown.

\textbf{Notes on $Y^{[2^3]}$.} 
The module remains indecomposable after removing the top and bottom copies of $D^{[6]}$ on the right, so some edges must connect to the $D^{[4,2]}$ on the right. Additionally, the quotient of the radical of $Y^{[2^3]}$ by its socle decomposes into two uniserial summands. This implies the existence of the edges between copies of $D^{[4,2]}$.

Now, removing only the top and bottom copies of $D^{[4,2]}$ from $Y^{[2^3]}$ leaves an indecomposable module, so some edges must connect to the copies of $D^{[6]}$ on the right. However, if we further quotient out both copies of $D^{[6]}$ from the socle of this module, the new module does decompose. Therefore the edges between copies of $D^{[6]}$ must be as shown.

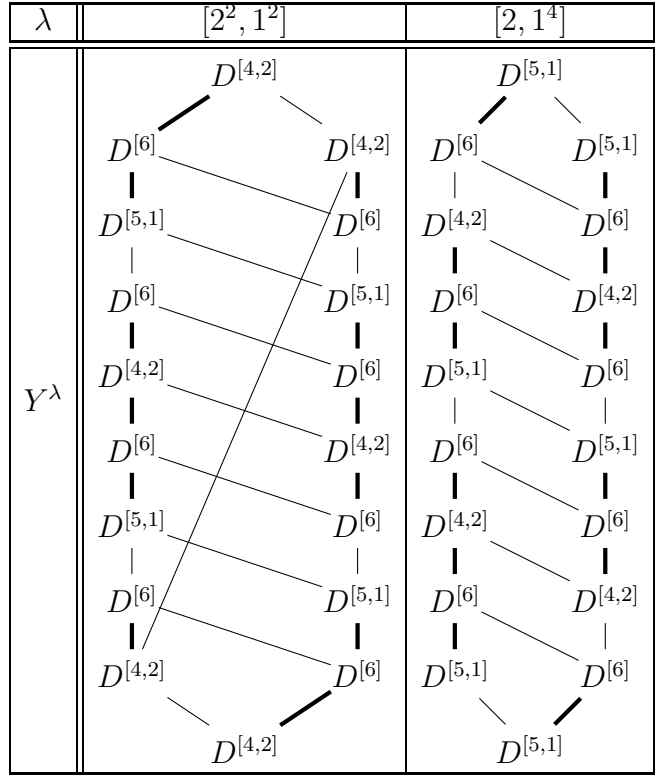
\begin{center}
\begin{tabular}{ |c||c|c| }
\hline
$\lambda$&$[2^2,1^2]$&$[2,1^4]$\\
\hline\hline
$Y^\lambda$ 
&
\begin{tikzpicture}[baseline=0pt]
\node[inner xsep=0.15] (10) at  (0,-4.5)  {$D^{[4,2]}$};
\node[inner xsep=0.15] (91) at  (-1.5,-3.5)   {$D^{[4,2]}$};
\node[inner xsep=0.15] (81) at  (-1.5,-2.5)   {$D^{[6]}$};
\node[inner xsep=0.15] (71) at  (-1.5,-1.5)   {$D^{[5,1]}$};
\node[inner xsep=0.15] (61) at  (-1.5,-.5)   {$D^{[6]}$};
\node[inner xsep=0.15] (51) at  (-1.5,.5)   {$D^{[4,2]}$};
\node[inner xsep=0.15] (41) at  (-1.5,1.5)   {$D^{[6]}$};
\node[inner xsep=0.15] (31) at  (-1.5,2.5)   {$D^{[5,1]}$};
\node[inner xsep=0.15] (21) at  (-1.5,3.5)   {$D^{[6]}$};
\node[inner xsep=0.15] (1) at  (0,4.5)  {$D^{[4,2]}$};
\node[inner xsep=0.15] (22) at  (1.5,3.5)   {$D^{[4,2]}$};
\node[inner xsep=0.15] (32) at  (1.5,2.5)   {$D^{[6]}$};
\node[inner xsep=0.15] (42) at  (1.5,1.5)   {$D^{[5,1]}$};
\node[inner xsep=0.15] (52) at  (1.5,.5)   {$D^{[6]}$};
\node[inner xsep=0.15] (62) at  (1.5,-.5)   {$D^{[4,2]}$};
\node[inner xsep=0.15] (72) at  (1.5,-1.5)   {$D^{[6]}$};
\node[inner xsep=0.15] (82) at  (1.5,-2.5)   {$D^{[5,1]}$};
\node[inner xsep=0.15] (92) at  (1.5,-3.5)   {$D^{[6]}$};

\begin{pgfonlayer}{bg}
    \draw[ultra thick] (1) -- (21) -- (31);
    \draw (31)-- (41);
    \draw[ultra thick](41)-- (51) -- (61) -- (71);
    \draw (71)-- (81);
    \draw[ultra thick](81)-- (91);
    \draw(91)-- (10);
    \draw[ultra thick](10)-- (92) -- (82);
    \draw (82)-- (72);
    \draw[ultra thick] (72) -- (62) -- (52) -- (42);
    \draw (42)-- (32);
    \draw[ultra thick] (32)-- (22);
    \draw (22)-- (1);
    \draw (22) -- (91);
    \draw (21) -- (32);
    \draw (31)--(42);
    \draw (41)--(52);
    \draw (51)--(62);
    \draw (61)--(72);
    \draw (71)--(82);
    \draw (81)--(92);
\end{pgfonlayer}
\end{tikzpicture}

&

\begin{tikzpicture}[baseline=0pt]

\node[inner xsep=0.15] (10) at  (0,-4.5)  {$D^{[5,1]}$};
\node[inner xsep=0.15] (91) at  (-1,-3.5)   {$D^{[5,1]}$};
\node[inner xsep=0.15] (81) at  (-1,-2.5)   {$D^{[6]}$};
\node[inner xsep=0.15] (71) at  (-1,-1.5)   {$D^{[4,2]}$};
\node[inner xsep=0.15] (61) at  (-1,-.5)   {$D^{[6]}$};
\node[inner xsep=0.15] (51) at  (-1,.5)   {$D^{[5,1]}$};
\node[inner xsep=0.15] (41) at  (-1,1.5)   {$D^{[6]}$};
\node[inner xsep=0.15] (31) at  (-1,2.5)   {$D^{[4,2]}$};
\node[inner xsep=0.15] (21) at  (-1,3.5)   {$D^{[6]}$};
\node[inner xsep=0.15] (1) at  (0,4.5)  {$D^{[5,1]}$};
\node[inner xsep=0.15] (22) at  (1,3.5)   {$D^{[5,1]}$};
\node[inner xsep=0.15] (32) at  (1,2.5)   {$D^{[6]}$};
\node[inner xsep=0.15] (42) at  (1,1.5)   {$D^{[4,2]}$};
\node[inner xsep=0.15] (52) at  (1,.5)   {$D^{[6]}$};
\node[inner xsep=0.15] (62) at  (1,-.5)   {$D^{[5,1]}$};
\node[inner xsep=0.15] (72) at  (1,-1.5)   {$D^{[6]}$};
\node[inner xsep=0.15] (82) at  (1,-2.5)   {$D^{[4,2]}$};
\node[inner xsep=0.15] (92) at  (1,-3.5)   {$D^{[6]}$};

\begin{pgfonlayer}{bg}
    \draw[ultra thick] (1) -- (21);
    \draw (21)-- (31);
    \draw[ultra thick] (31)-- (41) -- (51);
    \draw (51)-- (61);
    \draw[ultra thick] (61)-- (71) -- (81) -- (91);
    \draw (91)-- (10);
    \draw[ultra thick] (10)-- (92);
    \draw (92)-- (82);
    \draw[ultra thick] (82)-- (72) -- (62);
    \draw (62)-- (52);
    \draw[ultra thick] (52)-- (42) -- (32) -- (22);
    \draw (22) -- (1);
    \draw (21) -- (32);
    \draw (31)--(42);
    \draw (41)--(52);
    \draw (51)--(62);
    \draw (61)--(72);
    \draw (71)--(82);
    \draw (81)--(92);
\end{pgfonlayer}
\end{tikzpicture}
\\
\hline
\end{tabular}
\end{center}

\textbf{Notes on $Y^{[2^2,1^2]}$.} This module is the projective cover of $D^{[4,2]}$ because $[2^2,1^2]$ and $[4,2]$ are conjugate partitions. There is a homomorphism from $Y^{[2^3]}$ to $Y^{[2^2,1^2]}$ having kernel the socle $D^{[6]}$ of $Y^{[2^3]}$. The quotient of $Y^{[2^2,1^2]}$ by the image of this homomorphism is uniserial. This implies the right side of the diagram and the diagonal line, and that the only remaining edges are in the left uniserial quotient, and from that quotient down to the right. We now calculate with the sections $\Rad^iY^{[2^2,1^2]} / \Rad^{i+2}Y^{[2^2,1^2]}$ where $1\le i\le 7$, and show that they are all string modules by testing indecomposability of the quotients by the simple modules in their socles.  


\textbf{Notes on $Y^{[2,1^4]}$.} This module is constructed as a tensor product $Y^{[4,1^2]}\otimes U$ where $U$ is a 2-dimensional uniserial module whose composition factors are two copies of the trivial module $D^{[6]}$. Such a module $U$ may be constructed as a section of $Y^{[3,1^3]}$, or of $Y^{[2^3]}$. The tensor product is verified to be projective and indecomposable with top $D^{[5,1]}$. We deduce that it is the projective cover of $D^{[5,1]}$. This is also a description of $Y^{[2,1^4]}$ because the partition $[2,1^4]$ is column 2-regular, and so the Young module is the projective cover of the simple indexed by the conjugate partition (see \cite{Erd1}). Because the tensor product of the short exact sequence $0\to D^{[6]}\to U\to D^{[6]}\to 0$ with $Y^{[4,1^2]}$ is exact, it follows that $Y^{[2,1^4]}$ has a submodule isomorphic to the uniserial module $Y^{[4,1^2]}$, with quotient also isomorphic to this module. This establishes the outer edges in the diagram, and that the only remaining edges go from the quotient to the submodule. Finally, examining 2-step layers $\Rad^i Y^{[2,1^4]} / \Rad^{i+2} Y^{[2,1^4]}$ confirms that they are string modules as shown (see notes on $Y^{[2^2]}$, $Y^{[2^2,1^2]}$, and $Y^{[2,1^4]}$ for methods).

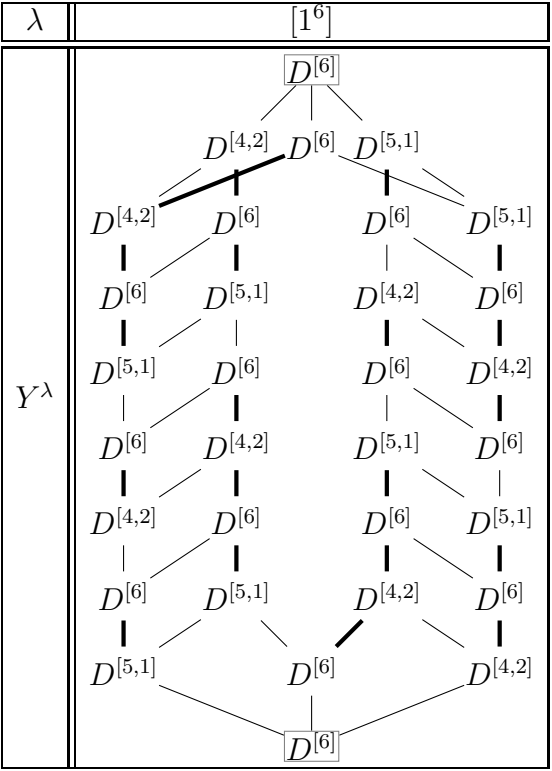
\begin{center}
\begin{tabular}{ |c||c| }
\hline
$\lambda$&$[1^6]$\\
\hline\hline
$Y^\lambda$ &
\begin{tikzpicture}[baseline=0pt]
\node[inner xsep=0.15] (53) at (1,0.5) {$D^{[6]}$};
\node[inner xsep=0.15] (52) at (-1,0.5) {$D^{[6]}$};
\node[inner xsep=0.15] (43) at (1,1.5) {$D^{[4,2]}$};
\node[inner xsep=0.15] (42) at (-1,1.5) {$D^{[5,1]}$};
\node[inner xsep=0.15] (33) at (1,2.5) {$D^{[6]}$};
\node[inner xsep=0.15] (32) at (-1,2.5) {$D^{[6]}$};
\node[inner xsep=0.15] (23) at (1,3.5) {$D^{[5,1]}$};
\node[inner xsep=0.15] (21) at (-1,3.5) {$D^{[4,2]}$};
\node[inner xsep=0.15] (63) at (1,-.5) {$D^{[5,1]}$};
\node[inner xsep=0.15] (62) at (-1,-.5) {$D^{[4,2]}$};
\node[inner xsep=0.15] (73) at (1,-1.5) {$D^{[6]}$};
\node[inner xsep=0.15] (72) at (-1,-1.5) {$D^{[6]}$};
\node[inner xsep=0.15] (83) at (1,-2.5) {$D^{[4,2]}$};
\node[inner xsep=0.15] (82) at (-1,-2.5) {$D^{[5,1]}$};

\node[inner xsep=0.15] (92) at (0,-3.5) {$D^{[6]}$};
\node[inner xsep=0.15] (93) at (2.5,-3.5) {$D^{[4,2]}$};
\node[inner xsep=0.15] (91) at (-2.5,-3.5) {$D^{[5,1]}$};
\node[draw=gray, inner sep=0.5] (101) at (0,-4.5) {$D^{[6]}$};
\node[inner xsep=0.15] (22) at (0,3.5) {$D^{[6]}$};
\node[draw=gray, inner sep=0.5] (11) at (0,4.5) {$D^{[6]}$};

\node[inner xsep=0.15] (54) at (2.5,0.5) {$D^{[4,2]}$};
\node[inner xsep=0.15] (51) at (-2.5,0.5) {$D^{[5,1]}$};
\node[inner xsep=0.15] (44) at (2.5,1.5) {$D^{[6]}$};
\node[inner xsep=0.15] (41) at (-2.5,1.5) {$D^{[6]}$};
\node[inner xsep=0.15] (34) at (2.5,2.5) {$D^{[5,1]}$};
\node[inner xsep=0.15] (31) at (-2.5,2.5) {$D^{[4,2]}$};
\node[inner xsep=0.15] (64) at (2.5,-.5) {$D^{[6]}$};
\node[inner xsep=0.15] (61) at (-2.5,-.5) {$D^{[6]}$};
\node[inner xsep=0.15] (74) at (2.5,-1.5) {$D^{[5,1]}$};
\node[inner xsep=0.15] (71) at (-2.5,-1.5) {$D^{[4,2]}$};
\node[inner xsep=0.15] (84) at (2.5,-2.5) {$D^{[6]}$};
\node[inner xsep=0.15] (81) at (-2.5,-2.5) {$D^{[6]}$};

\node () at (0,4.65) {};
\node () at (0,-4.65) {};

\begin{pgfonlayer}{bg}
    \draw (11) -- (21);
    \draw[ultra thick] (21)-- (32) -- (42);
    \draw (42)-- (52);
    \draw[ultra thick] (52)-- (62) -- (72) -- (82);
    \draw (82)-- (92) -- (101);
    \draw (11) -- (23);
    \draw[ultra thick] (23)-- (33);
    \draw (33)-- (43);
    \draw[ultra thick] (43)-- (53);
    \draw (53)-- (63);
    \draw[ultra thick] (63)-- (73) -- (83) -- (92);
    \draw (11) -- (22);
    \draw[ultra thick] (22) -- (31) -- (41) -- (51);
    \draw (51)-- (61);
    \draw[ultra thick] (61)-- (71);
    \draw (71)-- (81);
    \draw[ultra thick] (81)-- (91);
    \draw (91)-- (101);
    \draw (22) -- (34);
    \draw[ultra thick] (34) -- (44) -- (54) -- (64);
    \draw (64)-- (74);
    \draw[ultra thick] (74)-- (84) -- (93);
    \draw (93)-- (101);
    \draw (21) -- (31);
    \draw (32) -- (41);
    \draw (42) -- (51);
    \draw (52) -- (61);
    \draw (62) -- (71);
    \draw (72) -- (81);
    \draw (82) -- (91);
    \draw (23) -- (34);
    \draw (33) -- (44);
    \draw (43) -- (54);
    \draw (53) -- (64);
    \draw (63) -- (74);
    \draw (73) -- (84);
    \draw (83) -- (93);
\end{pgfonlayer}

\end{tikzpicture}

\\
\hline
\end{tabular}
\end{center}

\textbf{Notes on $Y^{[1^6]}$.} This module is the projective cover of $D^{[6]}$ (because $[1^6]$ and $[6]$ are conjugate partitions). It may be constructed as the tensor product $Y^{[3,1^3]}\otimes U$ where $U$ is the 2-dimensional uniserial module with two copies of $D^{[6]}$ as composition factors, used in the constructon of $Y^{[2,1^4]}$. This tensor product is verified to be projective, it has a copy of $D^{[6]}$ in its top layer, and it has the correct dimension, known from the Cartan matrix. 

The quotient $\Rad^2 Y^{[1^6]} / \Rad^8 Y^{[1^6]}$ decomposes as indicated by the diagram. We consider the 2-step layers of each of its two direct summands, and conclude that those sections are string modules as shown (see notes on $Y^{[2^2]}$, $Y^{[2^2,1^2]}$, and $Y^{[2,1^4]}$ for methods). There can be no more cross-diagonal edges within these summands because of the structures of $Y^{[2^2,1^2]}$ and $Y^{[2,1^4]}$, which are the projective covers of $D^{[4,2]}$ and $D^{[5,1]}$.
Finally, examination of the section  $\Rad^7 Y^{[1^6]} / \Rad^9 Y^{[1^6]}$ confirms that it is a string as shown.
The remaining edges in the diagram are established by duality.

\subsection{Young modules for $S_7$}
For this group we find in the upcoming calculations that there is a simple module admitting more than one non-split self-extension: $\dim\Ext_{\FF_2S_7}^1(D^{[5,2]},D^{[5,2]})=2$. Because we are working over $\FF_2$ there are exactly three non-split extensions, and we are able to distinguish between them as follows. One non-split extension occurs in the middle of the Young module $Y^{[5,1^2]}$, and we denote this by a dashed line. Another arises in the module $D^{[5,2]}\otimes U$, where $U$ is the 2-dimensional uniserial module with trivial composition factors, and we denote this with a dotted line. The third non-split extension, which is the sum of the other two, we denote by a dash-dotted line.

\noindent Young module summand multiplicities:

\begin{center}
{\renewcommand{\arraystretch}{1.2}
\scalebox{0.53}{
\begin{tabular}{ |c|c||c|c|c|c|c|c|c|c|c|c|c|c|c|c|c| }
\hline
\multicolumn{2}{|c||}{Table of}&\multicolumn{15}{|c|}{Permutation module}\\
\cline{3-17}
\multicolumn{2}{|c||}{$2$-Kostka numbers}&$[7]$&$[6,1]$&$[5,2]$&$[5,1^2]$&
$[4,3]$&$[4,2,1]$&$[4,1^3]$&$[3^2,1]$&$[3,2^2]$&$[3,2,1^2]$&$[3,1^4]$&$[2^3,1]$&$[2^2,1^3]$&$[2,1^5]$&$[1^7]$\\
\hline\hline
\multirow{15}{3em}{Young module} & $[7]$ & 1 & 1 & 1 & 0 & 1 & 1 & 0 & 0& 0 & 0 & 0 & 0 & 0 & 0 & 0\\
&$[6,1]$ & 0 & 1 & 1 & 2 & 0 & 1 & 2 & 0& 0 & 0 & 0 & 0 & 0 & 0 & 0 \\
&$[5,2]$ & 0 & 0 & 1 & 0 & 1 & 1 & 0 & 0 & 2 & 0 & 0 & 2 & 0 & 0 & 0 \\
&$[5,1^2]$ & 0 & 0 & 0 & 1 & 0 & 0 & 1 & 0& 0 & 0 & 0 & 0 & 0 & 0 & 0 \\
&$[4,3]$ & 0 & 0 & 0 & 0 & 1 & 1 & 0 & 2& 2 & 2 & 0 & 3 & 2 & 0 & 0 \\
&$[4,2,1]$ & 0 & 0 & 0 & 0 & 0 & 1 & 2 & 0& 1 & 2 & 4 & 3 & 4 & 4 & 0 \\
&$[4,1^3]$ & 0 & 0 & 0 & 0 & 0 & 0 & 1 & 0& 0 & 1 & 4 & 0 & 2 & 4 & 0 \\
&$[3^2,1]$ & 0 & 0 & 0 & 0 & 0 & 0 & 0 & 1& 0 & 1 & 0 & 0 & 1 & 0 & 0\\
&$[3,2^2]$ & 0 & 0 & 0 & 0 & 0 & 0 & 0 & 0& 1 & 0 & 0 & 1 & 0 & 0 & 0\\
&$[3,2,1^2]$ & 0 & 0 & 0 & 0 & 0 & 0 & 0 & 0& 0 & 1 & 2 & 2 & 5 & 10 & 20\\
&$[3,1^4]$ & 0 & 0 & 0 & 0 & 0 & 0 & 0 & 0& 0 & 0 & 1 & 0 & 0 & 1 & 0\\
&$[2^3,1]$ & 0 & 0 & 0 & 0 & 0 & 0 & 0 & 0& 0 & 0 & 0 & 1 & 2 & 4 & 8\\
&$[2^2,1^3]$ & 0 & 0 & 0 & 0 & 0 & 0 & 0 & 0& 0 & 0 & 0 & 0 & 1 & 4 & 14\\
&$[2,1^5]$ & 0 & 0 & 0 & 0 & 0 & 0 & 0 & 0& 0 & 0 & 0 & 0 & 0 & 1 & 6\\
&$[1^7]$ & 0 & 0 & 0 & 0 & 0 & 0 & 0 & 0& 0 & 0 & 0 & 0 & 0 & 0 & 1\\
\hline
\end{tabular}}}
\end{center}

\noindent Young module structure: 
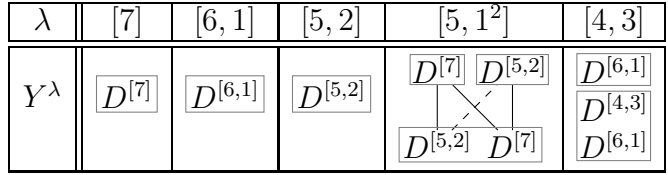
\begin{center}
\begin{tabular}{ |c||c|c|c|c|c| }
\hline
$\lambda$&$[7]$&$[6,1]$&$[5,2]$&$[5,1^2]$&$[4,3]$\\
\hline\hline
$Y^\lambda$&
\begin{tikzpicture}[baseline=0pt]
\node[inner sep=0.5] (1) at  (0,.175)  {$D^{[7]}$};
\node[draw=gray, fit = (1), inner sep=1] {};
\end{tikzpicture}&
\begin{tikzpicture}[baseline=0pt]
\node[inner sep=0.5] (1) at  (0,.175)  {$D^{[6,1]}$};
\node[draw=gray, fit = (1), inner sep=1] {};
\end{tikzpicture}&
\begin{tikzpicture}[baseline=0pt]
\node[inner sep=0.5] (1) at  (0,.175)  {$D^{[5,2]}$};
\node[draw=gray, fit = (1), inner sep=1] {};
\end{tikzpicture}
&
\begin{tikzpicture}[baseline=0pt,circ/.style={circle, draw=black, fill=white, ultra thick, inner xsep=0.01,minimum size=5mm}]

\node[draw=gray, inner sep=0.5] (1) at  (-.5,.5)   {$D^{[7]}$};
\node[draw=gray, inner sep=0.5] (2) at  (.5,.5)   {$D^{[5,2]}$};
\node[inner sep=0.15] (3) at  (-.5,-.5)  {$D^{[5,2]}$};
\node[inner sep=0.15] (4) at  (.5,-.5)   {$D^{[7]}$};
\node () at (0,.65) {};
\node[draw=gray, fit = (3) (4), inner sep=1] {};

\begin{pgfonlayer}{bg}
    \draw[dashed] (3)--(2);
    \draw (2)--(4)--(1)--(3);
\end{pgfonlayer}
\end{tikzpicture}

&

\begin{tikzpicture}[baseline=0pt]

\node[inner sep=0.15] (1) at  (0,.5)   {$D^{[6,1]}$};
\node[inner sep=0.15] (2) at  (0,0)   {$D^{[4,3]}$};
\node[inner sep=0.15] (3) at  (0,-.5)  {$D^{[6,1]}$};
\node[draw=gray, fit = (1), inner sep=.5] {};
\node[draw=gray, fit = (2) (3), inner sep=.5] {};

\end{tikzpicture}
\\[7mm]
\hline
\end{tabular}
\end{center}

\textbf{Notes on $Y^{[5,1^2]}$.} Quotienting out the socle $D^{[7]}$ yields an indecomposable module; instead quotienting out $D^{[5,2]}$ also yields an indecomposable module. Therefore the diagonal edges are as shown, and the radical and socle series require the diagram to be as shown.

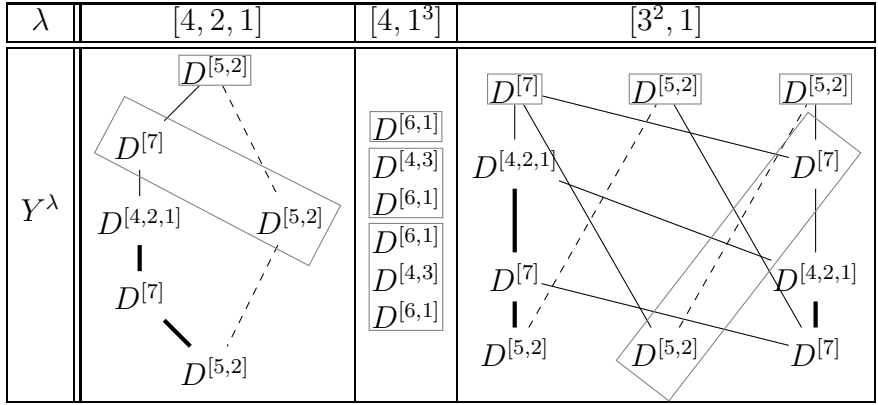
\begin{center}
\begin{tabular}{ |c||c|c|c| }
\hline
$\lambda$&$[4,2,1]$&$[4,1^3]$&$[3^2,1]$\\
\hline\hline
$Y^\lambda$&
\begin{tikzpicture}[baseline=0pt]
\node[draw=gray, inner sep=0.5] (1) at (0,2) {$D^{[5,2]}$};
\node (2) at (-1,1) {$D^{[7]}$};
\node[inner xsep=0.15] (3) at (-1,0) {$D^{[4,2,1]}$};
\node[inner xsep=0.15] (4) at (-1,-1) {$D^{[7]}$};
\node[inner xsep=0.15] (5) at (0,-2) {$D^{[5,2]}$};
\node (6) at (1,0) {$D^{[5,2]}$};
\node () at (0,2.15) {};
\draw[gray] (-1.6,.9)--(-1.2,1.65)--(1.67,.2)--(1.25,-.6)--(-1.6,.9);

\begin{pgfonlayer}{bg}
    \draw (1) -- (2) -- (3);
    \draw[ultra thick] (3)-- (4) -- (5);
    \draw[dashed] (1) -- (6) -- (5);
\end{pgfonlayer}

\end{tikzpicture}&

\begin{tikzpicture}[baseline=0pt]
\node[draw=gray, inner sep=0.5] (1) at  (0,1.25)   {$D^{[6,1]}$};
\node[inner sep=0.15] (2) at  (0,.75)   {$D^{[4,3]}$};
\node[inner sep=0.15] (3) at  (0,.25)  {$D^{[6,1]}$};
\node[inner sep=0.15] (4) at  (0,-.25)   {$D^{[6,1]}$};
\node[inner sep=0.15] (5) at  (0,-.75)   {$D^{[4,3]}$};
\node[inner sep=0.15] (6) at  (0,-1.25)  {$D^{[6,1]}$};
\node[draw=gray, fit = (2) (3), inner sep=.5] {};
\node[draw=gray, fit = (4) (5) (6), inner sep=.5] {};
\node () at (0,1.4) {};
\node () at (0,-1.4) {};
\end{tikzpicture}

&

\begin{tikzpicture}[baseline=0pt]
\node[draw=gray, inner sep=0.5] (l1) at (-2,1.75) {$D^{[7]}$};
\node[inner xsep=0.15] (l2) at (-2,.75) {$D^{[4,2,1]}$};
\node[inner xsep=0.15] (l3) at (-2,-.75) {$D^{[7]}$};
\node[inner xsep=0.15] (l4) at (-2,-1.75) {$D^{[5,2]}$};
\node[draw=gray, inner sep=0.5] (c1) at (0,1.75) {$D^{[5,2]}$};
\node[inner xsep=0.15] (c2) at (0,-1.75) {$D^{[5,2]}$};
\node[draw=gray, inner sep=0.5] (r1) at (2,1.75) {$D^{[5,2]}$};
\node[inner xsep=0.15] (r2) at (2,.75) {$D^{[7]}$};
\node[inner xsep=0.15] (r3) at (2,-.75) {$D^{[4,2,1]}$};
\node[inner xsep=0.15] (r4) at (2,-1.75) {$D^{[7]}$};

\draw[gray] (0.05,-2.4)--(2.6,.9)--(1.9,1.439)--(-.65,-1.861)--(.05,-2.4);

\begin{pgfonlayer}{bg}
    \draw (l1) -- (l2);
    \draw[ultra thick] (l2)-- (l3) -- (l4);
    \draw (r1) -- (r2) -- (r3);
    \draw[ultra thick] (r3)-- (r4);
    \draw (l1) -- (r2);
    \draw (l2) -- (r3);
    \draw (l3) -- (r4);
    \draw (l1) -- (c2);
    \draw (c1) -- (r4);
    \draw[dashed] (c2) -- (r1);
    \draw[dashed] (l4) -- (c1);
\end{pgfonlayer}

\end{tikzpicture}
\\
\hline
\end{tabular}
\end{center}

\textbf{Notes on $Y^{[4,2,1]}$.} The module decomposes once the top and socle copies of $D^{[5,2]}$ are removed.

\textbf{Notes on $Y^{[4,1^3]}$.} This module is the projective cover of $D^{[6,1]}$. It is a direct summand of the tensor product $Y^{[5,1^2]} \otimes D^{[6,1]}$.

\textbf{Notes on $Y^{[3^2,1]}$.} The module does not decompose when both socle copies of $D^{[5,2]}$ are removed, but $Y^{[3^2,1]} /  \Soc Y^{[3^2,1]}$ has the one remaining middle copy of $D^{[5,2]}$ as a direct summand. Therefore that copy of $D^{[5,2]}$ must connect to the socle $D^{[7]}$ of $Y^{[3^2,1]}$. The edges between copies of $D^{[4,2,1]}$ and $D^{[7]}$, as well as all vertical edges in the diagram, may be established by examining quotients of the form $\Rad^i Y^{[3^2,1]} / \Rad^{i+2} Y^{[3^2,1]}$ for $1 \leq i \leq 3$, which are all strings as shown in the diagram.

To demonstrate the absence of an edge between center copies of $D^{[5,2]}$, as well as the existence of the cross-diagonal edges between copies of $D^{[5,2]}$, we remove copies of $D^{[7]}$ from the top and socle of $Y^{[3^2,1]}$. The remaining module is indecomposable, but quotienting out the socle copy of $D^{[4,2,1]}$ causes it to decompose as indicated by the diagram.

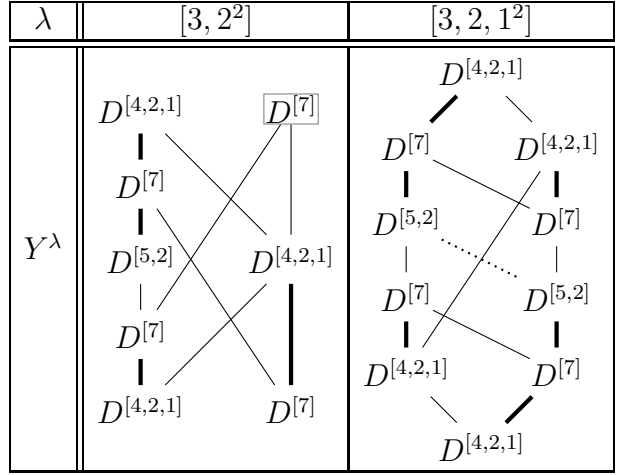
\begin{center}
\begin{tabular}{ |c||c|c| }
\hline
$\lambda$&$[3,2^2]$&$[3,2,1^2]$\\
\hline\hline
$Y^\lambda$&
\begin{tikzpicture}[baseline=0pt]
\node[inner xsep=0.15] (l1) at (-1,2) {$D^{[4,2,1]}$};
\node[inner xsep=0.15] (l2) at (-1,1) {$D^{[7]}$};
\node[inner xsep=0.15] (l3) at (-1,0) {$D^{[5,2]}$};
\node[inner xsep=0.15] (l4) at (-1,-1) {$D^{[7]}$};
\node[inner xsep=0.15] (l5) at (-1,-2) {$D^{[4,2,1]}$};
\node[draw=gray, inner sep=0.5] (r1) at (1,2) {$D^{[7]}$};
\node[inner xsep=0.15] (r2) at (1,0) {$D^{[4,2,1]}$};
\node[inner xsep=0.15] (r3) at (1,-2) {$D^{[7]}$};

\begin{pgfonlayer}{bg}
    \draw[ultra thick] (l1) -- (l2) -- (l3);
    \draw (l3)-- (l4);
    \draw[ultra thick] (l4)--(l5);
    \draw (r1) -- (r2);
    \draw[ultra thick] (r2)-- (r3);
    \draw (l1) -- (r2);
    \draw (l2) -- (r3);
    \draw (l4) -- (r1);
    \draw (l5) -- (r2);
\end{pgfonlayer}

\end{tikzpicture}

&

\begin{tikzpicture}[baseline=0pt]
\node[inner xsep=0.15] (top) at (0,2.5) {$D^{[4,2,1]}$};
\node[inner xsep=0.15] (l1) at (-1,1.5) {$D^{[7]}$};
\node[inner xsep=0.15] (l2) at (-1,.5) {$D^{[5,2]}$};
\node[inner xsep=0.15] (l3) at (-1,-.5) {$D^{[7]}$};
\node[inner xsep=0.15] (l4) at (-1,-1.5) {$D^{[4,2,1]}$};
\node[inner xsep=0.15] (r1) at (1,1.5) {$D^{[4,2,1]}$};
\node[inner xsep=0.15] (r2) at (1,.5) {$D^{[7]}$};
\node[inner xsep=0.15] (r3) at (1,-.5) {$D^{[5,2]}$};
\node[inner xsep=0.15] (r4) at (1,-1.5) {$D^{[7]}$};
\node[inner xsep=0.15] (bot) at (0,-2.5) {$D^{[4,2,1]}$};

\begin{pgfonlayer}{bg}
    \draw[ultra thick] (top)--(l1)--(l2);
    \draw (l2)--(l3);
    \draw[ultra thick] (l3)--(l4);
    \draw (l4)--(bot);
    \draw[ultra thick] (bot)--(r4)--(r3);
    \draw (r3)--(r2);
    \draw[ultra thick] (r2)--(r1);
    \draw (r1)--(top);
    \draw (l4)--(r1);
    \draw (l1)--(r2);
    \draw[thick,dotted] (l2)--(r3);
    \draw (l3)--(r4);
\end{pgfonlayer}

\end{tikzpicture}
\\
\hline
\end{tabular}
\end{center}

\textbf{Notes on $Y^{[3,2^2]}$.} When $D^{[4,2,1]}$, both copies of $D^{[7]}$, and $D^{[5,2]}$ are sequentially removed from the top of $Y^{[3,2^2]}$, the remaining module is a string as shown. Instead removing $D^{[7]}$ and then both copies of $D^{[4,2,1]}$ from the top of $Y^{[3,2^2]}$ yields an indecomposable module; further removing the top-left $D^{[7]}$ from this new module causes it to decompose, implying the existence of the edges between copies of $D^{[7]}$.

\textbf{Notes on $Y^{[3,2,1^2]}$.} This module is the projective cover of $D^{[4,2,1]}$. To confirm its structure, we first examine quotients of the form $\Rad^i Y^{[3,2,1^2]}$ $/\Rad^{i+2} Y^{[3,2,1^2]}$, which are all strings as shown in the diagram. We then isolate the second-highest and second-lowest copies of $D^{[4,2,1]}$ by removing all other simple modules from the top and bottom of $Y^{[3,2,1^2]}$; the radical series of this new module demonstrates that the two copies of $D^{[4,2,1]}$ are connected by the cross-diagonal edge shown.

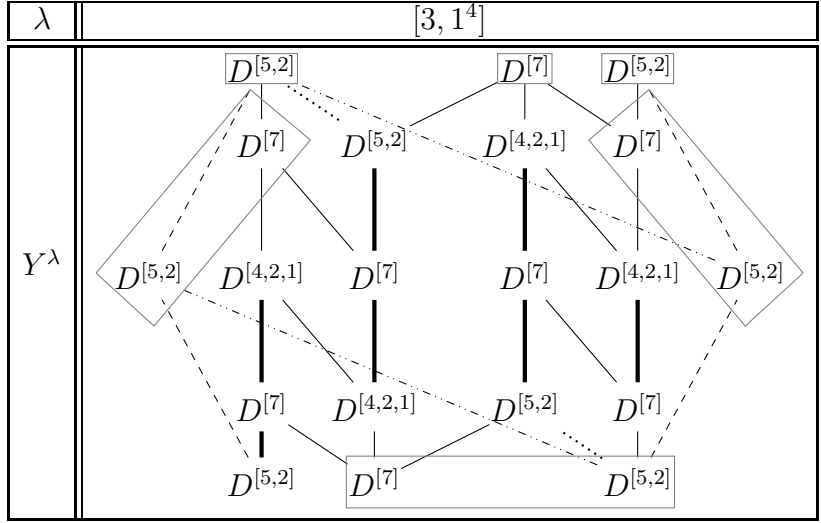
\begin{center}
\begin{tabular}{ |c||c| }
\hline
$\lambda$&$[3,1^4]$\\
\hline\hline
$Y^\lambda$&
\begin{tikzpicture}[baseline=0pt]
\node () at (0,2.9) {};
\node[draw=gray, inner sep=0.5] (11) at (2.5,2.75) {$D^{[5,2]}$};
\node[draw=gray, inner sep=0.5] (12) at (1,2.75) {$D^{[7]}$};
\node[draw=gray, inner sep=0.5] (13) at (-2.5,2.75) {$D^{[5,2]}$};
\node[inner xsep=0.15] (21) at (2.5,1.75) {$D^{[7]}$};
\node[inner xsep=0.15] (22) at (1,1.75) {$D^{[4,2,1]}$};
\node[inner xsep=0.15] (23) at (-1,1.75) {$D^{[5,2]}$};
\node[inner xsep=0.15] (24) at (-2.5,1.75) {$D^{[7]}$};
\node[inner xsep=0.15] (31) at (4,0) {$D^{[5,2]}$};
\node[inner xsep=0.15] (32) at (2.5,0) {$D^{[4,2,1]}$};
\node[inner xsep=0.15] (33) at (1,0) {$D^{[7]}$};
\node[inner xsep=0.15] (34) at (-1,0) {$D^{[7]}$};
\node[inner xsep=0.15] (35) at (-2.5,0) {$D^{[4,2,1]}$};
\node[inner xsep=0.15] (36) at (-4,0) {$D^{[5,2]}$};
\node[inner xsep=0.15] (41) at (2.5,-1.75) {$D^{[7]}$};
\node[inner xsep=0.15] (42) at (1,-1.75) {$D^{[5,2]}$};
\node[inner xsep=0.15] (43) at (-1,-1.75) {$D^{[4,2,1]}$};
\node[inner xsep=0.15] (44) at (-2.5,-1.75) {$D^{[7]}$};
\node[inner xsep=0.15] (51) at (2.5,-2.75) {$D^{[5,2]}$};
\node[inner xsep=0.15] (52) at (-1,-2.75) {$D^{[7]}$};
\node[inner xsep=0.15] (53) at (-2.5,-2.75) {$D^{[5,2]}$};
\node[inner xsep=0.15] () at (0,-3.1) {};

\draw[gray] (-4.7,0.03)--(-3.93,-.675)--(-1.85,1.8)--(-2.65,2.47)--(-4.7,0.03);

\draw[gray] (4.7,0.03)--(3.93,-.675)--(1.85,1.8)--(2.65,2.47)--(4.7,0.03);

\begin{pgfonlayer}{bg}
    \draw[dashed] (51)--(31)--(11);
    \draw[dashed] (53)--(36)--(13);
    \draw (11)--(21)--(32);
    \draw[ultra thick] (32)--(41);
    \draw (41)--(51);
    \draw[thick,dotted] (51)--(42);
    \draw[ultra thick] (42)--(33)--(22);
    \draw (22)--(12)--(23);
    \draw[ultra thick] (23)--(34)--(43);
    \draw (43)--(52)--(44);
    \draw (13)--(24)--(35);
    \draw[ultra thick] (35)--(44)--(53);
    \draw (21)--(12);
    \draw (32)--(22);
    \draw (41)--(33);
    \draw[thick,dotted] (23)--(13);
    \draw (34)--(24);
    \draw (43)--(35);
    \draw (42)--(52);
    \draw[dash dot dot] (31)--(13);
    \draw[dash dot dot] (36)--(51);
    \node[draw=gray, fit = (51) (52), inner sep=.5] {};
\end{pgfonlayer}

\end{tikzpicture}
\\
\hline
\end{tabular}
\end{center}

\textbf{Notes on $Y^{[3,1^4]}$.} Quotients of the form $\Rad^i Y^{[3,1^4]} / \Rad^{i+2} Y^{[3,1^4]}$ for $i=1, 2$ are easily confirmed to be direct sums of strings, and the structure of $\Rad Y^{[3,1^4]} / \Rad^4 Y^{[3,1^4]}$ follows from the way it decomposes.

We next examine $Y^{[3,1^4]} / \Soc^3 Y^{[3,1^4]}$, which is indecomposable. Removing both copies of $D^{[7]}$ from the bottom of this module yields the direct sum indicated by the diagram, and demonstrates that the bottom copy of $D^{[5,2]}$ is generated by one top copy of $D^{[5,2]}$ and one top copy of $D^{[7]}$. Instead, removing that bottom copy of $D^{[5,2]}$ causes the module to decompose into two direct summands as shown; therefore that $D^{[5,2]}$ must be generated by one simple module from each of those summands. These manipulations imply that the entire quotient module has the structure shown.

Finally, we investigate the edges that connect to the outer copies of $D^{[5,2]}$ on the left and right. We first construct a module $U$ with four composition factors $D^{[5,2]}$: the top two copies of $D^{[5,2]}$ and the two outer copies of $D^{[5,2]}$.
(This can be done by first constructing a uniserial module with composition factors $D^{[5,2]},D^{[7]},D^{[4,2,1]},D^{[7]}$ from $Y^{[4,2,1]}$. We quotient this uniserial module out of $Y^{[3,1^4]}$, and then remove excess simple modules from the module that results.) 
The module $U$ is indecomposable; this implies the existence of the outer edges from the top copies of $D^{[5,2]}$ to the outer copies, as well as either one or two cross-edges. Examination of $Y^{[3,1^4]} / \Soc^3 Y^{[3,1^4]}$, which decomposes after $D^{[7]}$ is removed from its top, confirms that there can be no other edges to the outer copies of $D^{[5,2]}$.

To examine the cross-edges, we next construct from $Y^{[5,1^2]}$ a uniserial module with two composition factors $D^{[5,2]}$, quotient this module out from $Y^{[3,1^4]}$, and remove excess simple modules until we are left with a module $V$ that has four composition factors $D^{[5,2]}$. These must be the top two copies of $D^{[5,2]}$, the middle copy of $D^{[5,2]}$, and one of the outer copies of $D^{[5,2]}$. To determine which outer copy, we note that $V$ decomposes as $(D^{[5,2]}/(D^{[5,2]}\oplus D^{[5,2]})) \oplus D^{[5,2]}$. If $V$ contained the left outer $D^{[5,2]}$ without a cross-edge, it would have decomposed as $(D^{[5,2]}/D^{[5,2]})\oplus (D^{[5,2]}/ D^{[5,2]})$; and if $V$ contained either outer $D^{[5,2]}$ with a cross-edge, it would have been a string and would not have decomposed. The given decomposition therefore implies that $V$ contains the right outer copy of $D^{[5,2]}$ with no cross-edge. Since $U$ is indecomposable, the other cross-edge must exist (between the outer left and top right copies of $D^{[5,2]}$), so $U$ is a string as shown.

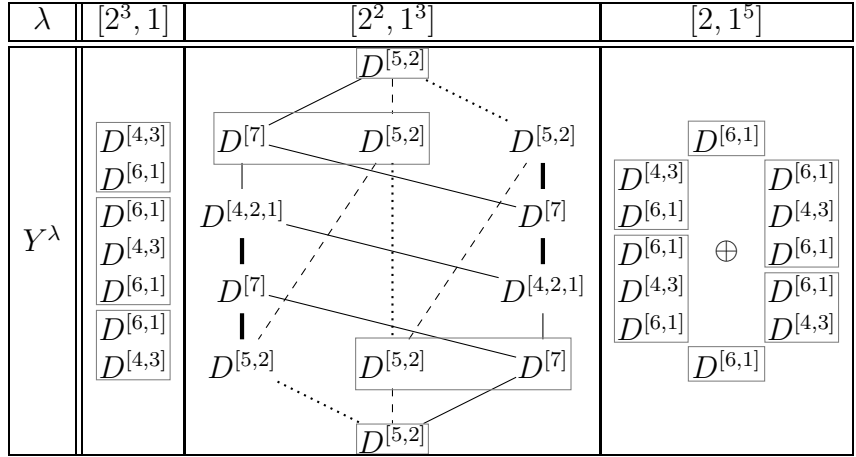
\begin{center}
\begin{tabular}{ |c||c|c|c| }
\hline
$\lambda$&$[2^3,1]$&$[2^2,1^3]$&$[2,1^5]$\\
\hline\hline
$Y^\lambda$&
\begin{tikzpicture}[baseline=0pt]
\node[inner sep=0.15] (1) at (0,1.5) {$D^{[4,3]}$};
\node[inner sep=0.15] (2) at (0,1) {$D^{[6,1]}$};
\node[inner sep=0.15] (3) at (0,.5) {$D^{[6,1]}$};
\node[inner sep=0.15] (4) at (0,0) {$D^{[4,3]}$};
\node[inner sep=0.15] (5) at (0,-.5) {$D^{[6,1]}$};
\node[inner sep=0.15] (6) at (0,-1) {$D^{[6,1]}$};
\node[inner sep=0.15] (7) at (0,-1.5) {$D^{[4,3]}$};
\node[draw=gray, fit = (1) (2), inner sep=.5] {};
\node[draw=gray, fit = (3) (4) (5), inner sep=.5] {};
\node[draw=gray, fit = (6) (7), inner sep=.5] {};
\node () at (0,1.65) {};
\node () at (0,-1.65) {};
\end{tikzpicture}

&

\begin{tikzpicture}[baseline=0pt]
\node[draw=gray, inner sep=0.5] (top) at (0,2.5) {$D^{[5,2]}$};
\node[inner xsep=0.15] (l1) at (-2,1.5) {$D^{[7]}$};
\node[inner xsep=0.15] (l2) at (-2,.5) {$D^{[4,2,1]}$};
\node[inner xsep=0.15] (l3) at (-2,-.5) {$D^{[7]}$};
\node[inner xsep=0.15] (l4) at (-2,-1.5) {$D^{[5,2]}$};
\node[inner xsep=0.15] (r1) at (2,1.5) {$D^{[5,2]}$};
\node[inner xsep=0.15] (r2) at (2,.5) {$D^{[7]}$};
\node[inner xsep=0.15] (r3) at (2,-.5) {$D^{[4,2,1]}$};
\node[inner xsep=0.15] (r4) at (2,-1.5) {$D^{[7]}$};
\node[inner xsep=0.15] (c1) at (0,1.5) {$D^{[5,2]}$};
\node[inner xsep=0.15] (c2) at (0,-1.5) {$D^{[5,2]}$};
\node[draw=gray, inner sep=0.5] (bot) at (0,-2.5) {$D^{[5,2]}$};
\begin{pgfonlayer}{bg}
    \draw (top) -- (l1) -- (l2);
    \draw[ultra thick] (l2)-- (l3) -- (l4);
    \draw[thick,dotted] (l4)-- (bot);
    \draw (bot)-- (r4) -- (r3);
    \draw[ultra thick] (r3)-- (r2) -- (r1);
    \draw[thick,dotted] (r1) -- (top);
    \draw (l1) -- (r2);
    \draw (l2) -- (r3);
    \draw (l3) -- (r4);
    \draw[dashed] (top) -- (c1);
    \draw[dotted,thick] (c1)-- (c2);
    \draw[dashed] (c2)-- (bot);
    \draw[dashed] (c1)-- (l4);
    \draw[dashed] (c2) -- (r1);
    \node[draw=gray, fit = (c2) (r4), inner sep=.5] {};
    \node[draw=gray, fit = (l1) (c1), inner sep=.5] {};
\end{pgfonlayer}
\end{tikzpicture}

&

\begin{tikzpicture}[baseline=0pt]
\node[inner sep=0.15] (top) at (0,1.5) {$D^{[6,1]}$};
\node[inner sep=0.15] (l1) at (-1,1) {$D^{[4,3]}$};
\node[inner sep=0.15] (l2) at (-1,.5) {$D^{[6,1]}$};
\node[inner sep=0.15] (l3) at (-1,0) {$D^{[6,1]}$};
\node[inner sep=0.15] (l4) at (-1,-.5) {$D^{[4,3]}$};
\node[inner sep=0.15] (l5) at (-1,-1) {$D^{[6,1]}$};
\node[inner sep=0.15] (r1) at (1,1) {$D^{[6,1]}$};
\node[inner sep=0.15] (r2) at (1,.5) {$D^{[4,3]}$};
\node[inner sep=0.15] (r3) at (1,0) {$D^{[6,1]}$};
\node[inner sep=0.15] (r4) at (1,-.5) {$D^{[6,1]}$};
\node[inner sep=0.15] (r5) at (1,-1) {$D^{[4,3]}$};
\node[inner sep=0.15] (bot) at (0,-1.5) {$D^{[6,1]}$};
\node[inner sep=0.15] (plus) at (0,0) {$\oplus$};
\node[draw=gray, fit = (top), inner sep=1] {};
\node[draw=gray, fit = (l1) (l2), inner sep=.5] {};
\node[draw=gray, fit = (l3) (l4) (l5), inner sep=.5] {};
\node[draw=gray, fit = (bot), inner sep=1] {};
\node[draw=gray, fit = (r5) (r4), inner sep=.5] {};
\node[draw=gray, fit = (r3) (r2) (r1), inner sep=.5] {};
\node () at (0,1.65) {};
\node () at (0,-1.65) {};

\end{tikzpicture}
\\
\hline
\end{tabular}
\end{center}

\textbf{Notes on $Y^{[2^3,1]}$.} This module is the projective cover of $D^{[4,3]}$. It is a direct summand of the tensor product $Y^{[3,2^2]} \otimes D^{[6,1]}$.

\textbf{Notes on $Y^{[2^2,1^3]}$.} This module is the projective cover of $D^{[5,2]}$. It is a direct summand of the tensor product $Y^{[3,2,1^2]} \otimes D^{[6,1]}$. To examine its structure, we first construct from $Y^{[5,1^2]}$ a uniserial module with two composition factors $D^{[5,2]}$. Removing this module from the top and bottom of $Y^{[2^2,1^3]}$ isolates the four leftmost and four rightmost simple modules; the structure of this outer ladder can be confirmed by previously explained methods (e.g. see notes on $Y^{[3,2,1^2]}$ and $Y^{[3,1^3]}$). Additionally, isolating the four bottommost copies of $D^{[5,2]}$ from the third socle allows the edges between them to be confirmed. We finally check that there are no additional edges to the center copies of $D^{[5,2]}$ by removing the three bottommost copies of $D^{[5,2]}$ and the topmost copy of $D^{[5,2]}$; this new module has the remaining center copy of $D^{[5,2]}$ as a direct summand.

\textbf{Notes on $Y^{[2,1^5]}$.} This module is the projective cover of $D^{[6,1]}$. It may be constructed as the tensor product $Y^{[4,1^3]}\otimes U$ where $U$ is the 2-dimensional uniserial module with two copies of $D^{[7]}$ as composition factors. Such a module $U$ may be constructed by removing copies of $D^{[5,2]}$ from the top and bottom of $Y^{[5,1^2]}$. This tensor product is verified to be projective and indecomposable, and it has a copy of $D^{[6,1]}$ in its top layer.

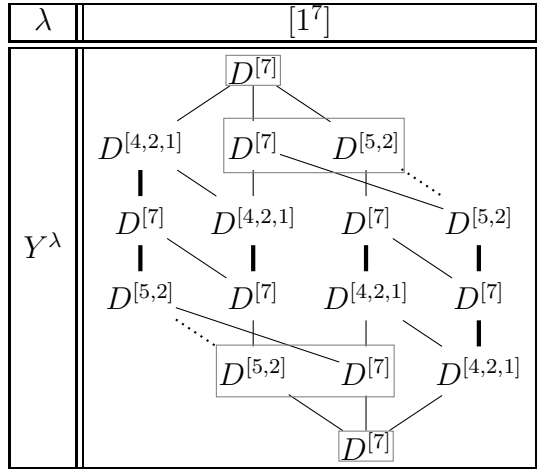
\begin{center}
\begin{tabular}{ |c||c| }
\hline
$\lambda$&$[1^7]$\\
\hline\hline
$Y^\lambda$&
\begin{tikzpicture}[baseline=0pt]
\node[draw=gray, inner sep=0.5] (1) at (-.75,2.5) {$D^{[7]}$};
\node[inner xsep=0.15] (21) at (-2.25,1.5) {$D^{[4,2,1]}$};
\node[inner xsep=0.15] (22) at (-.75,1.5) {$D^{[7]}$};
\node[inner xsep=0.15] (23) at (.75,1.5) {$D^{[5,2]}$};
\node[inner xsep=0.15] (31) at (-2.25,.5) {$D^{[7]}$};
\node[inner xsep=0.15] (32) at (-.75,.5) {$D^{[4,2,1]}$};
\node[inner xsep=0.15] (33) at (.75,.5) {$D^{[7]}$};
\node[inner xsep=0.15] (34) at (2.25,.5) {$D^{[5,2]}$};
\node[inner xsep=0.15] (41) at (-2.25,-.5) {$D^{[5,2]}$};
\node[inner xsep=0.15] (42) at (-.75,-.5) {$D^{[7]}$};
\node[inner xsep=0.15] (43) at (.75,-.5) {$D^{[4,2,1]}$};
\node[inner xsep=0.15] (44) at (2.25,-.5) {$D^{[7]}$};
\node[inner xsep=0.15] (51) at (-.75,-1.5) {$D^{[5,2]}$};
\node[inner xsep=0.15] (52) at (.75,-1.5) {$D^{[7]}$};
\node[inner xsep=0.15] (53) at (2.25,-1.5) {$D^{[4,2,1]}$};
\node[draw=gray, inner sep=0.5] (6) at (.75,-2.5) {$D^{[7]}$};
\node () at (0,2.65) {};
\node () at (0,-2.65) {};

\node[draw=gray, fit = (22) (23), inner sep=1] {};
\node[draw=gray, fit = (51) (52), inner sep=.5] {};

\begin{pgfonlayer}{bg}
    \draw (1)--(21);
    \draw[ultra thick](21)--(31)--(41);
    \draw[thick,dotted] (41)--(51);
    \draw (51)--(6);
    \draw (1)--(22)--(32);
    \draw[ultra thick] (32) --(42);
    \draw (42)--(51);
    \draw (1)--(23)--(33);
    \draw[ultra thick] (33)--(43);
    \draw (43)--(52)--(6);
    \draw[thick,dotted] (23)--(34);
    \draw[ultra thick] (34)--(44)--(53);
    \draw (53)--(6);
    \draw (21) -- (32);
    \draw (31) -- (42);
    \draw (22) -- (34);
    \draw (33) -- (44);
    \draw (43)-- (53);
    \draw (41) -- (52);
\end{pgfonlayer}
\end{tikzpicture}
\\
\hline
\end{tabular}
\end{center}

\textbf{Notes on $Y^{[1^7]}$.} This module is the projective cover of $D^{[7]}$. It is a direct summand of the tensor product $Y^{[2,1^5]} \otimes D^{[6,1]}$. To confirm the structure of its heart, we first examine 2-step radical layers, finding them to be strings and sums of strings as shown. We now perform additional checks to confirm that no other edges exist.

Removing copies of $D^{[7]}$ from the top and bottom of the heart yields the direct sum indicated by the diagram, which rules out the existence of any additional edges that do not connect to one of those copies of $D^{[7]}$. We eliminate the possibility of cross-diagonals in the ladders by taking the top three-step radical layer of the heart, removing $D^{[4,2,1]}$ from its top and $D^{[7]}$ from its bottom, and noticing that the leftmost copies of $D^{[7]}$ and $D^{[5,2]}$ are a (uniserial) direct summand of the resulting module. Finally, we construct two uniserial modules from $Y^{[4,2,1]}$. These are $U:=D^{[5,2]}/D^{[7]}/D^{[4,2,1]}/D^{[7]}$ and $V:=D^{[7]}/D^{[4,2,1]}/D^{[7]}/D^{[5,2]}$. Removing sequentially $U$ from the bottom and $V$ from the top of $Y^{[1^7]}$ yields a module that decomposes as $(D^{[7]}/D^{[4,2,1]}/D^{[7]}) \oplus (D^{[7]}/D^{[4,2,1]}/D^{[7]})$; this decomposition rules out the possibility of a connection between the two center columns of $Y^{[1^7]}$.

\section{The module structure of projective modules for Schur algebras in characteristic 2}
As noted previously, we compute with a basic algebra that is Morita equivalent to the usual Schur algebra, and this makes no difference to the structural properties that we consider.
By abuse of notation we denote the Morita equivalent algebra by the same symbols as are often used for the actual Schur algebra. Our algebra is
$$S_{\FF_2}(n,n)=\End_{\FF_2S_n}(\bigoplus_{\lambda\vdash n}Y^\lambda)
$$
and we store it on the computer as a matrix with rows and columns indexed by partitions $\lambda$, where the $(\lambda,\mu)$-entry is a list of matrices that is a basis for $\Hom_{\FF_2 S_n}(Y^\lambda,Y^\mu)$. The powers of the radical of $S_{\FF_2}(n,n)$ are stored in a similar way. Because the algebra is basic, the radical consists of those matrices that differ from $S_{\FF_2}(n,n)$ only on the diagonal, where the entries span the maximal ideal of $\End_{\FF_2S_n}(Y^\mu)$. This ideal can in turn be computed as the span of all composites $Y^\mu\to Y^\nu \to Y^\mu$ with $\mu\ne \nu$, using the fact that $S_{\FF_2}(n,n)$ is quasi-hereditary, so that no simple module extends itself. Thus this endomorphism ring and its radical powers are immediately computable from the Young modules that have been constructed. From this we obtain the Gabriel quiver of the Schur algebra.

The indecomposable projective $P(\mu)\cong \Hom_{\FF_2 S_n}(Y^\mu,\bigoplus_{\lambda\vdash n}Y^\lambda)$ is stored as row $\mu$ of the matrix that stores $S_{\FF_2}(n,n)$, and in this way we parametrize the indecomposable projectives $P(\lambda)$ and their unique simple quotients $L^{\lambda}$ by partitions of $n$. Its radical series arises as part of the computation of the radical series of  $S_{\FF_2}(n,n)$.
From the radical series we use the quasi-hereditary structure given by the dominance order on the partitions to deduce the standard (or Weyl) modules $\Delta^\lambda$. The diagrams for the projective modules are then built up from the standard modules, and we exploit the information coming from the Gabriel quiver, comparison of structures in different projective modules, and the fact that the projective modules are also injective precisely when the corresponding Young module $Y^\lambda$ is projective, and hence has a simple socle. This happens when $\lambda$ is column $p$-regular, and it means that the $\Delta^\mu$ appearing as a factor in such $P(\lambda)$ with $\mu$ earliest in the dominance order has a simple socle.

This latter result about injectivity is well known, but it is hard to find an exact reference consistent with our conventions, and for the convenience of the reader we prove the implication we need here. For any $\FF_2S_n$-module $Y$ we write $Y^\natural:=\Hom_{\FF_2 S_n}(Y,\bigoplus_{\lambda\vdash n}Y^\lambda)$ and we will use the fact that the functor $Y\mapsto Y^\natural$ is a contravariant equivalence between the full subcategory of $\FF_2S_n$-modules that are direct sums of the $Y^\lambda$, and the full subcategory of $S_{\FF_2}(n,n)$-modules whose objects are projective.

\begin{proposition}
Let $Y^\mu$ be a projective Young module for $\FF_2S_n$. Then $Y^{\mu\natural}$ is injective (and also projective) as a $S_{\FF_2}(n,n)$-module.
\end{proposition}

\begin{proof}
We have already been assuming the projectivity of $Y^{\mu\natural}$ and, indeed, it is a direct summand of $S_{\FF_2}(n,n)$. To show injectivity when $Y^\mu$ is projective, we show that $\Ext_{S_{\FF_2}(n,n)}^1(X,Y^{\mu\natural})=0$ for all $S_{\FF_2}(n,n)$-modules $X$. Take a $S_{\FF_2}(n,n)$-projective resolution 
$$
\cdots\to Y_2^\natural\to Y_1^\natural\to Y_0^\natural\to X\to 0
$$
where the $Y_i$ are direct sums of Young modules. All projectives may be written this way, by the equivalence of categories already referred to, and it arises by applying $\natural$ to a complex
$$
Y_0\to Y_1\to Y_2\to \cdots
$$
This latter complex is acyclic, except at 0, because $\bigoplus_{\lambda\vdash n}Y^\lambda$ has all indecomposable injective $\FF_2S_n$ modules as summands and the resolution was acyclic.  Consider the cochain complex
$$
\Hom_{S_{\FF_2}(n,n)}(Y_0^\natural, Y^{\mu\natural})\to 
\Hom_{S_{\FF_2}(n,n)}(Y_1^\natural, Y^{\mu\natural})\to 
\cdots
$$
whose cohomology computes the $\Ext_{S_{\FF_2}(n,n)}^i(X,Y^{\mu\natural})$. By the equivalence of categories, it is isomorphic to 
$$
\Hom_{\FF_2S_n}(Y^\mu,Y_0)\to 
\Hom_{\FF_2S_n}(Y^\mu,Y_1)\to 
\Hom_{\FF_2S_n}(Y^\mu,Y_2)\to \cdots
$$
and this is acyclic, except at 0, because $Y^\mu$ is projective. This shows that $\Ext_{S_{\FF_2}(n,n)}^1(X,Y^{\mu\natural})=0$ so that $Y^{\mu\natural}$ is injective.
\end{proof}

We display the standard or Weyl factors $\Delta$ in a quasi-hereditary filtration of the projective modules using similar conventions to the ones we used for Specht factors of Young modules: much of the time a $\Delta$ factor is indicated by joining its composition factors with thick lines, but sometimes we put the $\Delta$ factor in a box.

\subsection{Projective modules for $S_{\FF_2}(1,1)$}
There is one projective module $P[1]$, and it is simple.

\begin{center}
\begin{tabular}{ |c||c| }
\hline
$\lambda$&$[1]$\\
\hline\hline
$P(\lambda)$&
\begin{tikzpicture}[baseline=0pt]
\node[inner sep=0.15] (1) at  (0,0.2)  {$L^{[1]}$};
\node[draw=gray, fit = (1), inner sep=1] {};
\node () at (0,.4) {};
\node () at (0,0) {};
\end{tikzpicture}

\\
\hline
\end{tabular}
\end{center}

\subsection{Projective modules for $S_{\FF_2}(2,2)$}

\begin{center}
\begin{tabular}{ |c||c|c| }
\hline
$\lambda$&$[2]$&$[1^2]$\\
\hline\hline
$P(\lambda)$&

\begin{tikzpicture}[baseline=0pt]

\node[inner sep=0.15] (1) at  (0,.45)   {$L^{[2]}$};
\node[inner sep=0.15] (2) at  (0,-.05)  {$L^{[1^2]}$};

\node[draw=gray, fit = (1) (2), inner sep=.5] {};
\node () at (0,.9) {};
\node () at (0,-.5) {};
\end{tikzpicture}
&
\begin{tikzpicture}[baseline=0pt]

\node[inner sep=0.15] (1) at  (0,.7)   {$L^{[1^2]}$};
\node[inner sep=0.15] (2) at  (0,0.2)   {$L^{[2]}$};
\node[inner sep=0.15] (3) at  (0,-.3)  {$L^{[1^2]}$};

\node[draw=gray, fit = (1), inner sep=.5] {};
\node[draw=gray, fit = (2) (3), inner sep=.5] {};

\end{tikzpicture}

\\
\hline
\end{tabular}
\end{center}

\subsection{Projective modules for $S_{\FF_2}(3,3)$}

\begin{center}
\begin{tabular}{ |c||c|c|c| }
\hline
$\lambda$&$[3]$&$[2,1]$&$[1^3]$\\
\hline\hline
$P(\lambda)$&
\begin{tikzpicture}[baseline=0pt]
\node[inner sep=0.15] (1) at  (0,.45)   {$L^{[3]}$};
\node[inner sep=0.15] (2) at  (0,-.05)  {$L^{[1^3]}$};

\node[draw=gray, fit = (1) (2), inner sep=.5] {};
\node () at (0,.9) {};
\node () at (0,-.5) {};
\end{tikzpicture}
&
\begin{tikzpicture}[baseline=0pt]
\node[inner sep=0.5] (1) at  (0,.175)  {$L^{[2,1]}$};
\node[draw=gray, fit = (1), inner sep=1] {};
\end{tikzpicture}&

\begin{tikzpicture}[baseline=0pt]
\node[inner sep=0.15] (1) at  (0,.7)   {$L^{[1^3]}$};
\node[inner sep=0.15] (2) at  (0,0.2)   {$L^{[3]}$};
\node[inner sep=0.15] (3) at  (0,-.3)  {$L^{[1^3]}$};

\node[draw=gray, fit = (1), inner sep=.5] {};
\node[draw=gray, fit = (2) (3), inner sep=.5] {};
\end{tikzpicture}

\\
\hline
\end{tabular}
\end{center}

\subsection{Projective modules for $S_{\FF_2}(4,4)$}

\begin{center}
\begin{tabular}{ |c||c|c|c| }
\hline
$\lambda$&$[4]$&$[3,1]$&$[2^2]$\\
\hline
$P(\lambda)$ &
\begin{tikzpicture}[baseline=0pt]

\node[inner ysep=0.15] (1) at  (0,.75)   {$L^{[4]}$};
\node[inner sep=0.15] (2) at  (0,.25)   {$L^{[2^2]}$};
\node[inner sep=0.15] (3) at  (0,-.25)  {$L^{[3,1]}$};
\node[inner sep=0.15] (4) at  (0,-.75)   {$L^{[1^4]}$};

\node[draw=gray, fit = (1) (4), inner sep=.5] {};

\end{tikzpicture}

&

\begin{tikzpicture}[baseline=0pt]

\node[inner xsep=0.15] (top) at  (0,2.5)   {$L^{[3,1]}$};
\node[inner xsep=0.15] (l1) at  (-1,1.5)   {$L^{[2^2]}$};
\node[inner xsep=0.15] (l2) at  (-1,.5)  {$L^{[4]}$};
\node[inner xsep=0.15] (l3) at  (-1,-.5)   {$L^{[2^2]}$};
\node[inner xsep=0.15] (l4) at  (-1,-1.5)   {$L^{[3,1]}$};
\node[inner xsep=0.15] (bot) at  (0,-2.5)   {$L^{[1^4]}$};
\node[inner xsep=0.15] (r2) at  (1,-1)  {$L^{[2,1^2]}$};
\node[inner xsep=0.15] (r1) at  (1,1)   {$L^{[1^4]}$};

\begin{pgfonlayer}{bg}
    \draw (l1) -- (l2);
    \draw[ultra thick] (l2)-- (l3) -- (l4) -- (bot);
    \draw (bot) -- (r2);
    \draw[ultra thick] (r2)-- (r1) -- (top) -- (l1) -- (r2);
\end{pgfonlayer}
\end{tikzpicture}

& 
\begin{tikzpicture}[baseline=0pt]

\node[inner xsep=0.15] (top) at  (0,2)   {$L^{[2^2]}$};
\node[inner xsep=0.15] (l1) at  (-1.5,1)   {$L^{[4]}$};
\node[inner xsep=0.15] (l2) at  (-1.5,0)  {$L^{[2^2]}$};
\node[inner xsep=0.15] (l3) at  (-1.5,-1)   {$L^{[3,1]}$};
\node[inner xsep=0.15] (lbot) at  (-1.5,-2)   {$L^{[1^4]}$};
\node[inner xsep=0.15] (rbot) at  (1,-2)   {$L^{[2,1^2]}$};
\node[inner xsep=0.15] (c1) at  (0,.5)  {$L^{[2,1^2]}$};
\node[inner xsep=0.15] (c2) at  (0,-.5)   {$L^{[2^2]}$};
\node[inner xsep=0.15] (r1) at  (2,.5)  {$L^{[3,1]}$};
\node[inner xsep=0.15] (r2) at  (2,-.5)   {$L^{[1^4]}$};

\begin{pgfonlayer}{bg}
\draw (top) -- (l1);
\draw[ultra thick] (l1)--(l2)--(l3)--(lbot);
\draw (lbot)--(c1);
\draw[ultra thick] (c1)--(top);
\draw (top)--(r1);
\draw[ultra thick] (r1)--(r2)--(rbot)--(c2);
\draw (c2) --(c1);
\draw[ultra thick] (r1)--(c2);
\end{pgfonlayer}
\end{tikzpicture}
\\
\hline
\end{tabular}
\end{center}

\begin{center}
\begin{tabular}{ |c||c|c|c|c| }
\hline
$\lambda$&$[2,1^2]$&$[1^4]$\\
\hline
 $P(\lambda)$ 

&

\begin{tikzpicture}[baseline=0pt]

\node[inner xsep=0.15] (1) at  (0,2)   {$L^{[2,1^2]}$};
\node[inner xsep=0.15] (21) at  (-1,1)   {$L^{[2^2]}$};
\node[inner xsep=0.15] (22) at  (1,1)  {$L^{[1^4]}$};
\node[inner xsep=0.15] (31) at  (-2,0)   {$L^{[2,1^2]}$};
\node[inner xsep=0.15] (32) at  (0,0)   {$L^{[3,1]}$};
\node[inner xsep=0.15] (41) at  (-1,-1)   {$L^{[2^2]}$};
\node[inner xsep=0.15] (42) at  (1,-1)  {$L^{[1^4]}$};
\node[inner xsep=0.15] (5) at  (0,-2)   {$L^{[2,1^2]}$};

\begin{pgfonlayer}{bg}
    \draw (1)--(21)--(32)--(22);
    \draw[ultra thick] (22)--(1);
    \draw[ultra thick] (21)--(31);
    \draw (31)--(41);
    \draw[ultra thick] (41)--(32);
    \draw[ultra thick] (41)--(5)--(42)--(32);
\end{pgfonlayer}
\end{tikzpicture}

& 
\begin{tikzpicture}[baseline=0pt]

\node () at (0,3.15) {};
\node[draw=gray, inner sep=0.5] (top) at  (0,3)   {$L^{[1^4]}$};
\node[inner xsep=0.15] (l1) at  (-1.5,2)   {$L^{[2,1^2]}$};
\node[inner xsep=0.15] (l2) at  (-1.5,0)  {$L^{[1^4]}$};
\node[inner xsep=0.15] (l3) at  (-1.5,-2)   {$L^{[3,1]}$};
\node[inner xsep=0.15] (bot) at  (0,-3)   {$L^{[1^4]}$};
\node[inner xsep=0.15] (c1) at  (0,1)   {$L^{[2^2]}$};
\node[inner xsep=0.15] (c2) at  (0,0)  {$L^{[4]}$};
\node[inner xsep=0.15] (c3) at  (0,-1)   {$L^{[2^2]}$};
\node[inner xsep=0.15] (r1) at  (1.5,2)  {$L^{[3,1]}$};
\node[inner xsep=0.15] (r2) at  (1.5,0)   {$L^{[1^4]}$};
\node[inner xsep=0.15] (r3) at  (1.5,-2)   {$L^{[2,1^2]}$};

\begin{pgfonlayer}{bg}
\draw (top) -- (l1);
\draw[ultra thick] (l1)--(l2);
\draw (l2)--(l3);
\draw[ultra thick] (l3)--(bot);
\draw (l1)--(c3);
\draw[ultra thick] (c3)--(l3);
\draw (c1)--(c2);
\draw[ultra thick] (c2)--(c3);
\draw (top)--(r1);
\draw[ultra thick] (r1)--(c1)--(r3);
\draw (r3)--(bot);
\draw[ultra thick] (r1)--(r2)--(r3);
\end{pgfonlayer}
\end{tikzpicture}
\\
\hline
\end{tabular}
\end{center}

\subsection{Projective modules for $S_{\FF_2}(5,5)$}

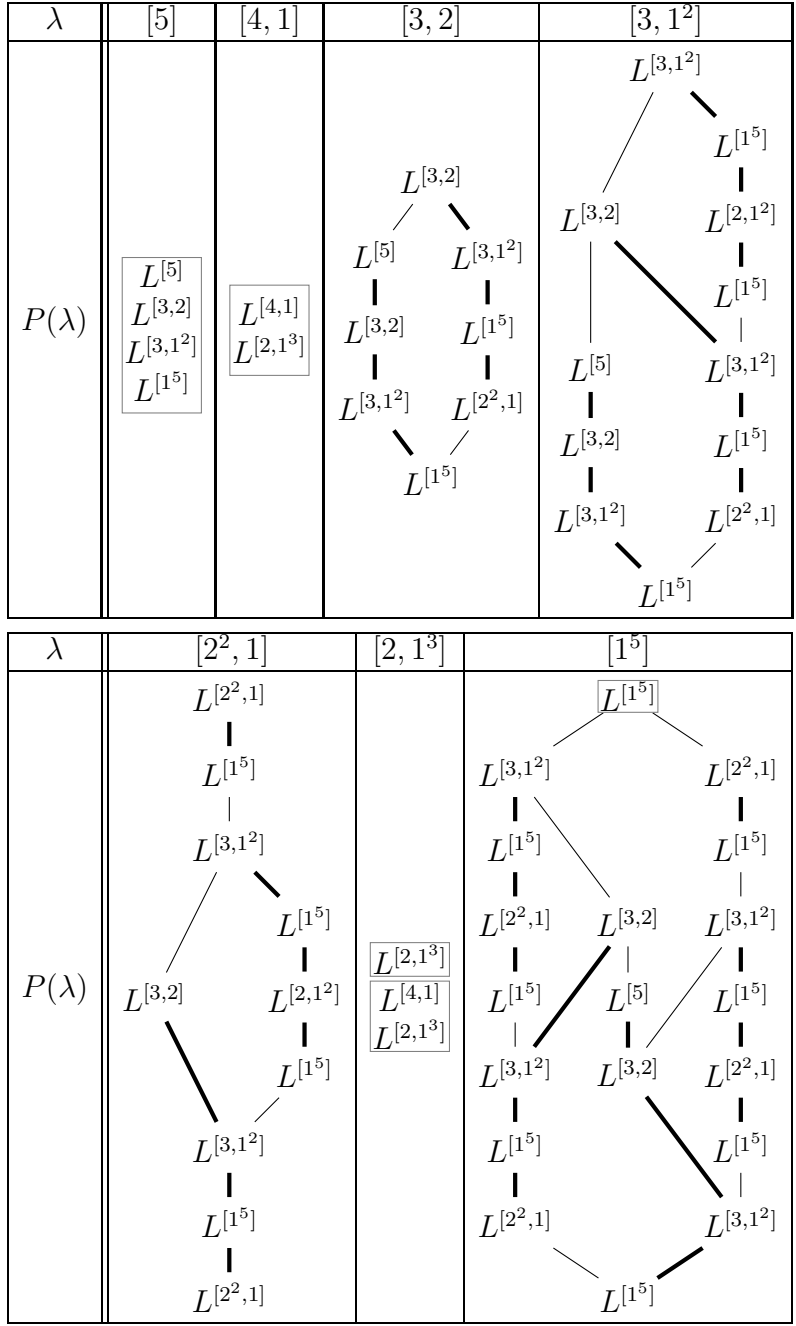
\begin{center}
\begin{tabular}{ |c||c|c|c|c| }
\hline
$\lambda$&$[5]$&$[4,1]$&$[3,2]$&$[3,1^2]$\\
\hline
 $P(\lambda)$ &
\begin{tikzpicture}[baseline=0pt]

\node[inner sep=0.15] (1) at  (0,.75)   {$L^{[5]}$};
\node[inner xsep=0.15] (2) at  (0,.25)   {$L^{[3,2]}$};
\node[inner xsep=0.15] (3) at  (0,-.25)  {$L^{[3,1^2]}$};
\node[inner xsep=0.15] (4) at  (0,-.75)   {$L^{[1^5]}$};

\node[draw=gray, fit = (1) (3) (4), inner sep=.5] {};

\end{tikzpicture}

&
\begin{tikzpicture}[baseline=0pt]

\node[inner xsep=0.15] (1) at  (0,.25)   {$L^{[4,1]}$};
\node[inner xsep=0.15] (2) at  (0,-.25)   {$L^{[2,1^3]}$};

\node[draw=gray, fit = (1) (2), inner sep=.5] {};

\end{tikzpicture}
&
\begin{tikzpicture}[baseline=0pt]

\node[inner xsep=0.15] (top) at  (0,2)   {$L^{[3,2]}$};
\node[inner xsep=0.15] (l1) at  (-.75,1)   {$L^{[5]}$};
\node[inner xsep=0.15] (l2) at  (-.75,0)  {$L^{[3,2]}$};
\node[inner xsep=0.15] (l3) at  (-.75,-1)   {$L^{[3,1^2]}$};
\node[inner xsep=0.15] (bot) at  (0,-2)   {$L^{[1^5]}$};
\node[inner xsep=0.15] (r3) at  (.75,-1)  {$L^{[2^2,1]}$};
\node[inner xsep=0.15] (r2) at  (.75,0)  {$L^{[1^5]}$};
\node[inner xsep=0.15] (r1) at  (.75,1)   {$L^{[3,1^2]}$};

\begin{pgfonlayer}{bg}
\draw (top) -- (l1);
\draw[ultra thick] (l1)--(l2)--(l3)--(bot);
\draw (bot)--(r3);
\draw[ultra thick] (r3)--(r2)--(r1)--(top);
\end{pgfonlayer}

\end{tikzpicture}

& 
\begin{tikzpicture}[baseline=0pt]

\node[inner xsep=0.15] (top) at  (0,3.5)   {$L^{[3,1^2]}$};
\node[inner xsep=0.15] (l1) at  (-1,1.5)   {$L^{[3,2]}$};
\node[inner xsep=0.15] (l2) at  (-1,-.5)  {$L^{[5]}$};
\node[inner xsep=0.15] (l3) at  (-1,-1.5)   {$L^{[3,2]}$};
\node[inner xsep=0.15] (l4) at  (-1,-2.5)   {$L^{[3,1^2]}$};
\node[inner xsep=0.15] (r1) at  (1,2.5)   {$L^{[1^5]}$};
\node[inner xsep=0.15] (r2) at  (1,1.5)  {$L^{[2,1^2]}$};
\node[inner xsep=0.15] (r3) at  (1,0.5)   {$L^{[1^5]}$};
\node[inner xsep=0.15] (r4) at  (1,-.5)   {$L^{[3,1^2]}$};
\node[inner xsep=0.15] (r5) at  (1,-1.5)   {$L^{[1^5]}$};
\node[inner xsep=0.15] (r6) at  (1,-2.5)   {$L^{[2^2,1]}$};
\node[inner xsep=0.15] (bot) at  (0,-3.5)   {$L^{[1^5]}$};

\begin{pgfonlayer}{bg}
\draw (top) -- (l1) --(l2);
\draw[ultra thick] (l2)--(l3)--(l4)--(bot);
\draw (bot)--(r6);
\draw[ultra thick] (r6)--(r5)--(r4);
\draw (r4)--(r3);
\draw[ultra thick] (r3)--(r2)--(r1)--(top);
\draw[ultra thick] (l1)--(r4);
\end{pgfonlayer}
\end{tikzpicture}
\\
\hline
\end{tabular}
\end{center}

\begin{center}
\begin{tabular}{ |c||c|c|c| }
\hline
$\lambda$&$[2^2,1]$&$[2,1^3]$&$[1^5]$\\
\hline
 $P(\lambda)$ 

&

\begin{tikzpicture}[baseline=0pt]

\node[inner xsep=0.15] (1) at  (0,4)   {$L^{[2^2,1]}$};
\node[inner xsep=0.15] (2) at  (0,3)   {$L^{[1^5]}$};
\node[inner xsep=0.15] (3) at  (0,2)  {$L^{[3,1^2]}$};
\node[inner xsep=0.15] (l) at  (-1,0)   {$L^{[3,2]}$};
\node[inner xsep=0.15] (4) at  (1,1)   {$L^{[1^5]}$};
\node[inner xsep=0.15] (5) at  (1,0)   {$L^{[2,1^2]}$};
\node[inner xsep=0.15] (6) at  (1,-1)  {$L^{[1^5]}$};
\node[inner xsep=0.15] (7) at  (0,-2)   {$L^{[3,1^2]}$};
\node[inner xsep=0.15] (8) at  (0,-3)   {$L^{[1^5]}$};
\node[inner xsep=0.15] (9) at  (0,-4)  {$L^{[2^2,1]}$};

\begin{pgfonlayer}{bg}
    \draw[ultra thick] (1)--(2);
    \draw (2)--(3);
    \draw[ultra thick] (3)--(4)--(5)--(6);
    \draw (6)--(7);
    \draw (3)--(l);
    \draw[ultra thick](l)--(7)--(8)--(9);
\end{pgfonlayer}
\end{tikzpicture}
&
\begin{tikzpicture}[baseline=0pt]

\node[inner sep=0.15] (1) at  (0,.5)   {$L^{[2,1^3]}$};
\node[inner sep=0.15] (2) at  (0,0)   {$L^{[4,1]}$};
\node[inner sep=0.15] (3) at  (0,-.5)  {$L^{[2,1^3]}$};

\node[draw=gray, fit = (1), inner sep=.5] {};
\node[draw=gray, fit = (2) (3), inner sep=.5] {};

\end{tikzpicture}
& 
\begin{tikzpicture}[baseline=0pt]

\node () at (0,4.15) {};
\node[draw=gray, inner sep=0.5] (top) at  (0,4)   {$L^{[1^5]}$};
\node[inner xsep=0.15] (l1) at  (-1.5,3)   {$L^{[3,1^2]}$};
\node[inner xsep=0.15] (l2) at  (-1.5,2)  {$L^{[1^5]}$};
\node[inner xsep=0.15] (l3) at  (-1.5,1)   {$L^{[2^2,1]}$};
\node[inner xsep=0.15] (l4) at  (-1.5,0)   {$L^{[1^5]}$};
\node[inner xsep=0.15] (l5) at  (-1.5,-1)  {$L^{[3,1^2]}$};
\node[inner xsep=0.15] (l6) at  (-1.5,-2)   {$L^{[1^5]}$};
\node[inner xsep=0.15] (l7) at  (-1.5,-3)   {$L^{[2^2,1]}$};
\node[inner xsep=0.15] (bot) at  (0,-4)   {$L^{[1^5]}$};
\node[inner xsep=0.15] (c1) at  (0,1)   {$L^{[3,2]}$};
\node[inner xsep=0.15] (c2) at  (0,0)  {$L^{[5]}$};
\node[inner xsep=0.15] (c3) at  (0,-1)   {$L^{[3,2]}$};
\node[inner xsep=0.15] (r1) at  (1.5,3)  {$L^{[2^2,1]}$};
\node[inner xsep=0.15] (r2) at  (1.5,2)   {$L^{[1^5]}$};
\node[inner xsep=0.15] (r3) at  (1.5,1)   {$L^{[3,1^2]}$};
\node[inner xsep=0.15] (r4) at  (1.5,0)  {$L^{[1^5]}$};
\node[inner xsep=0.15] (r5) at  (1.5,-1)   {$L^{[2^2,1]}$};
\node[inner xsep=0.15] (r6) at  (1.5,-2)   {$L^{[1^5]}$};
\node[inner xsep=0.15] (r7) at  (1.5,-3)  {$L^{[3,1^2]}$};

\begin{pgfonlayer}{bg}
\draw (top) -- (l1);
\draw[ultra thick] (l1)--(l2)--(l3)--(l4);
\draw (l4)--(l5);
\draw[ultra thick] (l5)--(l6)--(l7);
\draw (l7)--(bot);
\draw (top) -- (r1);
\draw[ultra thick] (r1)--(r2);
\draw (r2)--(r3);
\draw[ultra thick] (r3)--(r4) --(r5)--(r6);
\draw (r6)--(r7);
\draw[ultra thick] (r7)--(bot);
\draw (l1)--(c1);
\draw[ultra thick] (c1)--(l5);
\draw (r3)--(c3);
\draw[ultra thick] (c3)--(r7);
\draw (c1) -- (c2);
\draw[ultra thick] (c2)--(c3);
\end{pgfonlayer}
\end{tikzpicture}
\\
\hline
\end{tabular}
\end{center}

\subsection{The Schur algebra $S_{\FF_2}(6,6)$}
We do not give the full structure of the projective modules in this case as the larger diagrams are complicated, with several tens of composition factors and elaborate Ext structure, so that the point of the diagram as a way to understand the structure of the modules is defeated. As partial information we give diagrams for the standard modules $\Delta^\lambda$ that arise in its quasi-hereditary structure, also known as the Weyl modules. The composition factor multiplicities in these modules are the decomposition numbers and this information is given in \cite{Gra} and \cite{Jam3}. We note that the decomposition matrix $D$ satisfies $D^TD=C$, the Cartan matrix.

\noindent Weyl module structure:
\begin{center}
\begin{tabular}{ |c||c|c|c| }
\hline
$\lambda$&$[6]$&$[5,1]$&$[4,2]$\\
\hline\hline
$\Delta^\lambda$&

\begin{tikzpicture}[baseline=0pt, xscale=1.3,yscale=1.3]
\node[inner sep=0.15] (1) at  (0,1.4)  {$L^{[6]}$};
\node[inner sep=0.15] (2) at  (-0.6,0.8)   {$L^{[5,1]}$};
\node[inner sep=0.15] (3) at  (0.6,0.8)   {$L^{[2^3]}$};
\node[inner sep=0.15] (4) at  (0,0.2)  {$L^{[3^2]}$};
\node[inner sep=0.15] (5) at  (0,-0.3)  {$L^{[3,1^3]}$};
\node[inner sep=0.15] (6) at  (0,-0.8)  {$L^{[1^6]}$};
\draw (1) -- (2) -- (4);
\draw (1) -- (3) -- (4);
\draw (4) -- (5);
\draw (5) -- (6);
\end{tikzpicture}

&

\begin{tikzpicture}[baseline=0pt, xscale=1.3,yscale=1.3]
\node[inner sep=0.15] (1) at  (0,1.4)  {$L^{[5,1]}$};
\node[inner sep=0.15] (2) at  (0,0.8)   {$L^{[3^2]}$};
\node[inner sep=0.15] (3) at  (-0.6,0.2)   {$L^{[3,1^3]}$};
\node[inner sep=0.15] (4) at  (0.6,0.2)  {$L^{[4,2]}$};
\node[inner sep=0.15] (5) at  (-0.6,-0.4)   {$L^{[1^6]}$};
\node[inner sep=0.15] (6) at  (0.6,-0.4)  {$L^{[4,1^2]}$};
\node[inner sep=0.15] (7) at  (0,-1)  {$L^{[2,1^4]}$};
\draw (1) -- (2) -- (3) -- (5) -- (7);
\draw (2) -- (4) -- (6) -- (7);
\node () at (0,1.6) {};
\node () at (0,-1.3) {};
\end{tikzpicture}

&

\begin{tikzpicture}[baseline=0pt, xscale=1.3,yscale=1.3]
\node[inner sep=0.15] (1) at  (0,1.4)  {$L^{[4,2]}$};
\node[inner sep=0.15] (2) at  (-0.6, 0.8)   {$L^{[4,1^2]}$};
\node[inner sep=0.15] (3) at  (0.6, 0.8)   {$L^{[3^2]}$};
\node[inner sep=0.15] (4) at  (-1.2,0.2)  {$L^{[2,1^4]}$};
\node[inner sep=0.15] (5) at  (0,0.2)   {$L^{[3,1^3]}$};
\node[inner sep=0.15] (6) at  (1.2,0.2)  {$L^{[2^3]}$};
\node[inner sep=0.15] (7) at  (-0.6, -0.4)  {$L^{[1^6]}$};
\node[inner sep=0.15] (8) at  (0, -1)   {$L^{[2^2,1^2]}$};
\draw (1) -- (2) -- (4) -- (7) -- (8);
\draw (2) -- (5) -- (7);
\draw (1) -- (3) -- (5);
\draw (3) -- (6) -- (8);
\draw (2) -- (6);
\end{tikzpicture}

\\
\hline
\end{tabular}
\end{center}

\begin{center}
\begin{tabular}{ |c||c|c|c|c| }
\hline
$\lambda$&$[4,1^2]$&$[3^2]$&$[3,2,1]$&$[3,1^3]$\\
\hline\hline
$\Delta^\lambda$&

\begin{tikzpicture}[baseline=0pt, xscale=1.3,yscale=1.3]
\node[inner sep=0.15] (1) at  (0,1.4)  {$L^{[4,1^2]}$};
\node[inner sep=0.15] (2) at  (-0.8,0.8)  {$L^{[2,1^4]}$};
\node[inner sep=0.15] (3) at  (0,0.8)   {$L^{[3,1^3]}$};
\node[inner sep=0.15] (4) at  (0.8,0.8)  {$L^{[2^3]}$};
\node[inner sep=0.15] (5) at  (0, 0.2)  {$L^{[1^6]}$};
\node[inner sep=0.15] (6) at  (0, -0.4)   {$L^{[2^2,1^2]}$};
\node[inner sep=0.15] (7) at  (0, -1)   {$L^{[1^6]}$};
\node () at (0,1.6) {};
\node () at (0,-1.2) {};
\draw (1) -- (3) -- (5) -- (6) -- (7);
\draw (1) -- (4) -- (6);
\draw (1) -- (2) -- (5);
\end{tikzpicture}

&

\begin{tikzpicture}[baseline=0pt, xscale=1.3,yscale=1.3]
\node[inner sep=0.15] (2) at  (0,1.1)   {$L^{[3^2]}$};
\node[inner sep=0.15] (3) at  (-0.6,0.5)   {$L^{[3,1^3]}$};
\node[inner sep=0.15] (4) at  (0.6,0.2)  {$L^{[2^3]}$};
\node[inner sep=0.15] (5) at  (-0.6,-0.1)   {$L^{[1^6]}$};
\node[inner sep=0.15] (7) at  (0,-0.7)  {$L^{[2^2,1^2]}$};
\draw(2) -- (3) -- (5) -- (7);
\draw (2) -- (4)  -- (7);
\end{tikzpicture}

&

\begin{tikzpicture}[baseline=0pt, xscale=1.3,yscale=1.3]
\node[inner sep=0.15] (3) at  (0,0.2)  {$L^{[3,2,1]}$};
\end{tikzpicture}

&

\begin{tikzpicture}[baseline=0pt, xscale=1.3,yscale=1.3]
\node[inner sep=0.15] (1) at  (0,1.2)  {$L^{[3,1^3]}$};
\node[inner sep=0.15] (2) at  (0,0.7)  {$L^{[1^6]}$};
\node[inner sep=0.15] (3) at  (0,0.2)   {$L^{[2^2,1^2]}$};
\node[inner sep=0.15] (4) at  (0,-0.3)  {$L^{[1^6]}$};
\node[inner sep=0.15] (5) at  (0,-0.8)  {$L^{[2,1^4]}$};
\node () at (0,1.4) {};
\node () at (0,-1) {};
\draw(1) -- (2) -- (3) -- (4) -- (5);
\end{tikzpicture}

\\
\hline
\end{tabular}
\end{center}

\begin{center}
\begin{tabular}{ |c||c|c|c|c| }
\hline
$\lambda$&$[2^3]$&$[2^2,1^2]$& $[2,1^4]$& $[1^6]$\\
\hline\hline
$\Delta^\lambda$&

\begin{tikzpicture}[baseline=0pt, xscale=1.3,yscale=1.3]
\node[inner sep=0.15] (1) at  (0,0.7)  {$L^{[2^3]}$};
\node[inner sep=0.15] (2) at  (0,0.2)   {$L^{[2^2,1^2]}$};
\node[inner sep=0.15] (3) at  (0,-0.3)  {$L^{[1^6]}$};
\node () at (0,0.9) {};
\node () at (0,-0.6) {};
\draw(1) -- (2) -- (3);
\end{tikzpicture}

&

\begin{tikzpicture}[baseline=0pt, xscale=1.3,yscale=1.3]
\node[inner sep=0.15] (1) at  (0,0.7)  {$L^{[2^2,1^2]}$};
\node[inner sep=0.15] (2) at  (0,0.2)   {$L^{[1^6]}$};
\node[inner sep=0.15] (3) at  (0,-0.3)  {$L^{[2,1^4]}$};
\draw(1) -- (2) -- (3);
\end{tikzpicture}

&

\begin{tikzpicture}[baseline=0pt, xscale=1.3,yscale=1.3]
\node[inner sep=0.15] (1) at  (0,0.4)   {$L^{[2,1^4]}$};
\node[inner sep=0.15] (2) at  (0,-0.1)  {$L^{[1^6]}$};
\draw(1) -- (2);
\end{tikzpicture}

&

\begin{tikzpicture}[baseline=0pt, xscale=1.3,yscale=1.3]
\node[inner sep=0.15] (3) at  (0,0.2)  {$L^{[1^6]}$};
\end{tikzpicture}

\\
\hline
\end{tabular}
\end{center}

The Gabriel quiver of $S_{\FF_2}(6,6)$ is the directed graph whose vertices are the isomorphism types of the simple modules, and where the number of edges from a simple $S$ to a simple $T$ is $\dim\Ext_{S_{\FF_2}(6,6)}^1(S,T)$. For Schur algebras this relation on $S$ and $T$ is symmetric because there is a duality on modules under which simple modules are self dual. This means that for each edge $S\to T$ there is also an edge $T\to S$.

Gabriel quiver:
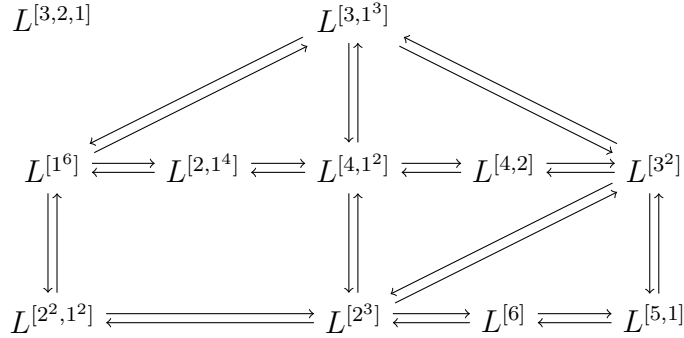
\begin{center}
\begin{tikzpicture}[node distance=4cm, scale=2]

\node (A) at (3, 0) {$L^{[6]}$};
\node (B) at (4, 0) {$L^{[5,1]}$};
\node (C) at (3, 1) {$L^{[4,2]}$};
\node (D) at (2, 1) {$L^{[4,1^2]}$};
\node (E) at (4, 1) {$L^{[3^2]}$};
\node (F) at (0,2) {$L^{[3,2,1]}$};
\node (G) at (2, 2) {$L^{[3,1^3]}$};
\node (H) at (2, 0) {$L^{[2^3]}$};
\node (I) at (0,0) {$L^{[2^2,1^2]}$};
\node (K) at (1, 1) {$L^{[2,1^4]}$};
\node (L) at (0, 1) {$L^{[1^6]}$};
\draw[->, transform canvas={yshift=.06cm}] (I) edge (H) (H) edge (A) (A) edge (B);
\draw[->, transform canvas={yshift=-.06cm}] (H) edge (I) (A) edge (H) (B) edge (A);
\draw[->, transform canvas={yshift=.06cm}] (L) edge (K) (K) edge (D) (D) edge (C)  (C) edge (E);
\draw[->, transform canvas={yshift=-.06cm}] (K) edge (L) (D) edge (K) (C) edge (D)  (E) edge (C);
\draw[->, transform canvas={xshift=.06cm}] (I) edge (L) (H) edge (D) (B) edge (E)  (D) edge (G);
\draw[->, transform canvas={xshift=-.06cm}] (L) edge (I) (D) edge (H) (E) edge (B)  (G) edge (D);
\draw[->, transform canvas={xshift=.03cm, yshift=-.06cm}] (L) edge (G) (H) edge (E) ;
\draw[->, transform canvas={xshift=-.03cm, yshift=.06cm}] (G) edge (L) (E) edge (H) ;
\draw[->, transform canvas={xshift=.03cm, yshift=.06cm}] (E) edge (G) ;
\draw[->, transform canvas={xshift=-.03cm, yshift=-.06cm}] (G) edge (E);
\end{tikzpicture}
\end{center}

\subsection{The Schur algebra $S_{\FF_2}(7,7)$}

Because $S_7$ has two blocks over $\FF_2$, so does $S_{\FF_2}(7,7)$. We will call the block corresponding to the principal block of $\FF_2S_7$ block 1, and the other one block 2. We have seen that the non-principal block of $\FF_2S_7$ is Morita equivalent to the principal block of $\FF_2S_5$ in such a way that the diagrams for the Young modules in these blocks are the same, under a correspondence of partitions
$$
\begin{matrix}
[6,1]\\ [4,3]\\ [2^3,1] \\ [2,1^5]\\
\end{matrix}
\leftrightarrow
\begin{matrix}
[5]\\ [3,2]\\ [2^2,1] \\ [1^5]\\
\end{matrix}.
$$
Accordingly, block 2 of $S_{\FF_2}(7,7)$ is Morita equivalent to block 1 of $S_{\FF_2}(5,5)$, and the diagrams for projective modules are the same in these two blocks, after applying this correspondence of partitions. For this reason we omit the diagrams for block 2.

Block 1 of $S_{\FF_2}(7,7)$ is harder to describe and we give less complete information. As with $S_{\FF_2}(6,6)$ we give diagrams for the standard modules and the Gabriel quiver. The comments made about these structures in the context of $S_{\FF_2}(6,6)$ also apply here.

\noindent Weyl module structure for block 1:
\begin{center}
\begin{tabular}{ |c||c|c|c|c| }
\hline
$\lambda$&$[7]$&$[5,2]$&$[5,1^2]$&$[4,2,1]$\\
\hline\hline
$\Delta^\lambda$&

\begin{tikzpicture}[baseline=0pt, xscale=1.3,yscale=1.3]
\node[inner sep=0.15] (1) at  (0,1.2)  {$L^{[7]}$};
\node[inner sep=0.15] (2) at  (-0.5,0.7)   {$L^{[5,1^2]}$};
\node[inner sep=0.15] (3) at  (0.5,0.7)   {$L^{[3,2^2]}$};
\node[inner sep=0.15] (4) at  (0,0.2)  {$L^{[3^2,1]}$};
\node[inner sep=0.15] (5) at  (0,-0.3)  {$L^{[3,1^4]}$};
\node[inner sep=0.15] (6) at  (0,-0.8)  {$L^{[1^7]}$};
\draw (1) -- (2) -- (4);
\draw (1) -- (3) -- (4);
\draw (4) -- (5);
\draw (5) -- (6);
\end{tikzpicture}

&

\begin{tikzpicture}[baseline=0pt, xscale=1.3,yscale=1.3]
\node[inner sep=0.15] (1) at  (0,2)  {$L^{[5,2]}$};
\node[inner sep=0.15] (2) at  (0,1.5)  {$L^{[3^2,1]}$};
\node[inner sep=0.15] (3) at  (0,1.0)   {$L^{[5,1^2]}$};
\node[inner sep=0.15] (4) at  (0,0.5)  {$L^{[3^2,1]}$};
\node[inner sep=0.15] (5) at  (0,0)  {$L^{[3,1^4]}$};
\node[inner sep=0.15] (6) at  (0,-0.5)  {$L^{[4,2,1]}$};
\node[inner sep=0.15] (7) at  (0,-1)  {$L^{[3,1^4]}$};
\node[inner sep=0.15] (8) at  (0,-1.5)  {$L^{[2^2,1^3]}$};
\node () at (0,2.2) {};
\node () at (0,-1.7) {};
\draw(1) -- (2) -- (3) -- (4) -- (5) -- (6) -- (7) -- (8);
\end{tikzpicture}

&

\begin{tikzpicture}[baseline=0pt, xscale=1.3,yscale=1.3]
\node[inner sep=0.15] (3) at  (0,1.5)   {$L^{[5,1^2]}$};
\node[inner sep=0.15] (4) at  (0,1)  {$L^{[3^2,1]}$};
\node[inner sep=0.15] (5) at  (0,0.5)  {$L^{[3,1^4]}$};
\node[inner sep=0.15] (6) at  (-0.5,0)  {$L^{[4,2,1]}$};
\node[inner sep=0.15] (7) at  (-0.5,-0.5)  {$L^{[3,1^4]}$};
\node[inner sep=0.15] (8) at  (-0.5,-1)  {$L^{[2^2,1^3]}$};
\node[inner sep=0.15] (9) at  (0.5,0)  {$L^{[1^7]}$};
\draw (3) -- (4) -- (5) -- (6) -- (7) -- (8);
\draw (5) -- (9);
\end{tikzpicture}

&

\begin{tikzpicture}[baseline=0pt, xscale=1.3,yscale=1.3]
\node[inner sep=0.15] (1) at  (0,1.5)   {$L^{[4,2,1]}$};
\node[inner sep=0.15] (2) at  (0,1)  {$L^{[3,1^4]}$};
\node[inner sep=0.15] (3) at  (-0.5,0.5)  {$L^{[3^2,1]}$};
\node[inner sep=0.15] (4) at  (0.5,0.5)  {$L^{[2^2,1^3]}$};
\node[inner sep=0.15] (5) at  (0.5,0)  {$L^{[3,1^4]}$};
\node[inner sep=0.15] (6) at  (-0.5,-0.5)  {$L^{[3,2^2]}$};
\node[inner sep=0.15] (7) at  (0.5,-0.5)  {$L^{[1^7]}$};
\node[inner sep=0.15] (8) at  (0,-1)  {$L^{[3,2,1^2]}$};
\draw (1) -- (2) -- (3) -- (6)  -- (8);
\draw (2) -- (4) -- (5) -- (7)  -- (8);
\draw (3) -- (5);
\end{tikzpicture}

\\
\hline
\end{tabular}
\end{center}

\begin{center}
\begin{tabular}{ |c||c|c|c|c|c|c| }
\hline
$\lambda$&$[3^2,1]$&$[3,2^2]$&$[3,2,1^2]$&$[3,1^4]$&$[2^2,1^3]$& $[1^6]$\\
\hline\hline
$\Delta^\lambda$

&

\begin{tikzpicture}[baseline=0pt, xscale=1.3,yscale=1.3]
\node[inner sep=0.15] (2) at  (0,1)   {$L^{[3^2,1]}$};
\node[inner sep=0.15] (3) at  (-0.5,0.5)   {$L^{[3,1^4]}$};
\node[inner sep=0.15] (4) at  (0.5,0.25)  {$L^{[3,2^2]}$};
\node[inner sep=0.15] (5) at  (-0.5,0)   {$L^{[1^7]}$};
\node[inner sep=0.15] (7) at  (0,-0.5)  {$L^{[3,2,1^2]}$};
\node () at (0,1.2) {};
\node () at (0,-0.7) {};
\draw(2) -- (3) -- (5) -- (7);
\draw (2) -- (4)  -- (7);
\end{tikzpicture}

&

\begin{tikzpicture}[baseline=0pt, xscale=1.3,yscale=1.3]
\node[inner sep=0.15] (1) at  (0,0.7)  {$L^{[3,2^2]}$};
\node[inner sep=0.15] (2) at  (0,0.2)   {$L^{[3,2,1^2]}$};
\node[inner sep=0.15] (3) at  (0,-0.3)  {$L^{[1^7]}$};
\node () at (0,0.9) {};
\node () at (0,-0.6) {};
\draw(1) -- (2) -- (3);
\end{tikzpicture}

&

\begin{tikzpicture}[baseline=0pt, xscale=1.3,yscale=1.3]
\node[inner sep=0.15] (1) at  (0,1)  {$L^{[3,2,1^2]}$};
\node[inner sep=0.15] (2) at  (0,0.5)  {$L^{[1^7]}$};
\node[inner sep=0.15] (3) at  (0,0)  {$L^{[3,1^4]}$};
\node[inner sep=0.15] (4) at  (0,-0.5)  {$L^{[2^2,1^3]}$};
\draw (1) -- (2) -- (3) -- (4) ;
\end{tikzpicture}

&

\begin{tikzpicture}[baseline=0pt, xscale=1.3,yscale=1.3]
\node[inner sep=0.15] (1) at  (0,0.5)   {$L^{[3,1^4]}$};
\node[inner sep=0.15] (2) at  (-0.5,0)  {$L^{[2^2,1^3]}$};
\node[inner sep=0.15] (3) at  (0.5,0)  {$L^{[1^7]}$};
\draw(1) -- (2);
\draw(1) -- (3);
\end{tikzpicture}

&

\begin{tikzpicture}[baseline=0pt, xscale=1.3,yscale=1.3]
\node[inner sep=0.15] (3) at  (0,0.25)  {$L^{[2^2,1^3]}$};
\end{tikzpicture}

&

\begin{tikzpicture}[baseline=0pt, xscale=1.3,yscale=1.3]
\node[inner sep=0.15] (3) at  (0,0.25)  {$L^{[1^7]}$};
\end{tikzpicture}

\\
\hline
\end{tabular}
\end{center}

\noindent Gabriel quiver for block 1:
\begin{center}
    \begin{tikzpicture}[node distance=4cm, scale=2]
\node (A) at (0, 0) {$L^{[7]}$};
\node (B) at (1,2) {$L^{[5,2]}$};
\node (C) at (0, 1) {$L^{[5,1^2]}$};
\node (D) at (4, 1) {$L^{[4,2,1]}$};
\node (E) at (1, 1) {$L^{[3^2,1]}$};
\node (F) at (1,0) {$L^{[3,2^2]}$};
\node (G) at (2, 0) {$L^{[3,2,1^2]}$};
\node (H) at (3,1) {$L^{[3,1^4]}$};
\node (I) at (3,2) {$L^{[2^2,1^3]}$};
\node (J) at (3,0) {$L^{[1^7]}$};
\draw[->, transform canvas={yshift=.06cm}] (A) edge (F) (F) edge (G) (G) edge (J);
\draw[->, transform canvas={yshift=-.06cm}] (J) edge (G) (G) edge (F) (F) edge (A);
\draw[->, transform canvas={yshift=.06cm}] (C) edge (E) (E) edge (H) (H) edge (D);
\draw[->, transform canvas={yshift=-.06cm}] (D) edge (H) (H) edge (E) (E) edge (C);
\draw[->, transform canvas={xshift=.06cm}] (A) edge (C);
\draw[->, transform canvas={xshift=-.06cm}] (C) edge (A);
\draw[->, transform canvas={xshift=.06cm}] (F) edge (E) (E) edge (B);
\draw[->, transform canvas={xshift=-.06cm}] (B) edge (E) (E) edge (F);
\draw[->, transform canvas={xshift=.06cm}] (J) edge (H) (H) edge (I);
\draw[->, transform canvas={xshift=-.06cm}] (I) edge (H) (H) edge (J);
\end{tikzpicture}
\end{center}

\end{document}